\documentclass[10pt,draft,reqno]{amsart}
\usepackage[latin9]{inputenc}
\usepackage{color}
\usepackage{amsmath}
\usepackage{amsthm}
\usepackage{amssymb}
\usepackage{esint}

\makeatletter
\numberwithin{equation}{section} 
\numberwithin{figure}{section} 
\numberwithin{equation}{section}
\numberwithin{figure}{section}
\theoremstyle{plain}
\newtheorem{thm}{\protect\theoremname}[section]
  \theoremstyle{plain}
  \newtheorem{assumption}[thm]{\protect\assumptionname}
  \theoremstyle{remark}
  \newtheorem{rem}[thm]{\protect\remarkname}
  \theoremstyle{plain}
  \newtheorem{lem}[thm]{\protect\lemmaname}
  \theoremstyle{plain}
  \newtheorem{prop}[thm]{\protect\propositionname}
  \theoremstyle{plain}
  \newtheorem{cor}[thm]{\protect\corollaryname}

\usepackage{color}
\usepackage{latexsym}
\usepackage{bm}
\usepackage{graphicx}
\usepackage{wrapfig}
\usepackage{fancybox}

     \def\section{\@startsection{section}{1}%
     \z@{.7\linespacing\@plus\linespacing}{.5\linespacing}%
     {\bfseries
     \centering
     }}
     \def\@secnumfont{\bfseries}

\newcommand{\Rd}{\mathbb{R}^{d}}
\newcommand{\hs}{\mathbb{R}_+ \times \mathbb{R}^{d}}

\newcommand{\BN}{\mathbb{N}}

\newcommand{\BR}{\mathbb{R}}

\newcommand{\bfE}{\mathbf{E}}

  \providecommand{\corollaryname}{Corollary}
  \providecommand{\lemmaname}{Lemma}
  \providecommand{\propositionname}{Proposition}
  \providecommand{\remarkname}{Remark}
\providecommand{\theoremname}{Theorem}
\providecommand{\assumptionname}{Assumption}

  \providecommand{\corollaryname}{Corollary}
  
  \providecommand{\lemmaname}{Lemma}
  \providecommand{\propositionname}{Proposition}
  \providecommand{\remarkname}{Remark}
\providecommand{\theoremname}{Theorem}

  \providecommand{\corollaryname}{Corollary}
  
  \providecommand{\lemmaname}{Lemma}
  \providecommand{\propositionname}{Proposition}
  \providecommand{\remarkname}{Remark}
\providecommand{\theoremname}{Theorem}

  \providecommand{\corollaryname}{Corollary}
  
  \providecommand{\lemmaname}{Lemma}
  \providecommand{\propositionname}{Proposition}
  \providecommand{\remarkname}{Remark}
\providecommand{\theoremname}{Theorem}

\makeatother

  \providecommand{\assumptionname}{Assumption}
  \providecommand{\corollaryname}{Corollary}
  \providecommand{\lemmaname}{Lemma}
  \providecommand{\propositionname}{Proposition}
  \providecommand{\remarkname}{Remark}
\providecommand{\theoremname}{Theorem}

\begin{document}
\title[Martingale problem for perturbations of Lévy-type generators]{Well-posedness of the martingale problem for non-local perturbations of  Lévy-type generators}
\author[P. Jin]{Peng Jin}

\address{Peng Jin: Fakultät für Mathematik und Naturwissenschaften, Bergische
Universität Wuppertal, 42119 Wuppertal, Germany}

\email{jin@uni-wuppertal.de}



\begin{abstract} Let $L$ be a Lévy-type generator whose Lévy measure
is controlled from below by that of a non-degenerate $\alpha$-stable
($0<\alpha<2$) process. In this paper, we study the martingale problem
for the operator $\mathcal{L}_{t}=L+K_{t}$, with $K_{t}$ being a
time-dependent non-local operator defined by
\[
K_{t}f(x):=\int_{\mathbb{R}^{d}\backslash\{0\}}[f(x+y)-f(x)-\mathbf{1}_{\alpha>1}\mathbf{1}_{\{|y|\le1\}}y\cdot\nabla f(x)]M(t,x,dy),
\]
where $M(t,x,\cdot)$ is a Lévy measure on $\mathbb{R}^{d}\backslash\{0\}$
for each $(t,x)\in\hs$. We show that if
\[
\sup_{t\geq0,x\in\mathbb{R}^{d}}\int_{\mathbb{R}^{d}\backslash\{0\}}1\wedge|y|^{\beta}M(t,x,dy)<\infty
\]
for some $0<\beta<\alpha$, then the martingale problem for $\mathcal{L}_{t}$
is well-posed. \end{abstract}

\maketitle

\section{Introduction}

As a generalization of the fractional Laplacian $\triangle^{\alpha/2}$
($0<\alpha<2)$, the anisotropic fractional Laplacian is defined by
\begin{align*}
Af(x) & =\int_{\mathbb{R}^{d}\backslash\{0\}}\left[f(x+y)-f(x)-\mathbf{1}_{\{|y|\le1\}}y\cdot\nabla f(x)\right]\nu(dy),
\end{align*}
where
\[
\nu(B)=\int_{\mathbb{S}^{d-1}}\mu(d\xi)\int_{0}^{\infty}\mathbf{1}_{B}(r\xi)\frac{dr}{r^{1+\alpha}},\quad\forall B\in\mathcal{B}(\mathbb{R}^{d}),
\]
and $\mu$ a is a finite measure on $\mathbb{S}^{d-1}$. We call $\nu$
the Lévy measure and $\mu$ the sprectral measure of $A$. Clearly
the behaivior of the anisotropic fractional Laplacian is solely determined
by its spectral measure. Since $\mu$ can be any finite measure on
$\mathbb{S}^{d-1}$, this leads to some interesting properties of
$A$ that the fractional Laplacian $\triangle^{\alpha/2}$ does not
possess. As an example, the heat kernel of $A$ may have very different
type of estimates compared to $\triangle^{\alpha/2}$, see \cite{MR2286060}.

The anisotropic fractional Laplacian $A$ corresponds to a Markov
process, namely, it is the generator of an $\alpha$-stable process.
It is natural to ask the following question of stability: if we add
a small perturbation $B$ to $A$, does $A+B$ still correspond to
a Markov process, or more precisely, is the martingale problem for
$A+B$ well-posed? This problem has been well-studied when $1<\alpha<2$
and the perturbation operator $B$ is of drift-type $B=b(t,\cdot)\cdot\nabla$.
Depending on the regularity of the spectral measure $\mu$, various
classes of drifts $b$ have been introduced such that the martingale
problem for $A+b(t,\cdot)\cdot\nabla$ is well-posed. If $\mu$ is
the surface measure on $\mathbb{S}^{d-1}$, drifts belonging to the
Kato class $\mathcal{K}_{\alpha-1}^{d}$ were considered in \cite{MR3192504,chen2013uniqueness};
for the case when $\mu$ is non-degenerate, drifts from some Hölder
or $L^{p}$ spaces were treated in \cite{MR2945756,MR3127913,jin2015weak}.

In addition to drift-type perturbations mentioned above, perturbations
of $A$ including a lower order non-local term have also been investigated.
This type of perturbation was first considered in \cite{MR736974}.
There, the perturbation operator $B$ took the form
\[
Bf(x)=\mathbf{1}_{\alpha>1}b(x)\cdot\nabla f(x)+\int_{\mathbb{R}^{d}\backslash\{0\}}[f(x+y)-f(x)-\mathbf{1}_{\alpha>1}\mathbf{1}_{\{|y|\le1\}}y\cdot\nabla f(x)]M(x,dy),
\]
and, under some appropriate conditions on $\mu$, $b$ and $M$, uniqueness
of the martingale problem for $A+B$ was obtained. As an essential
step, some non-local estimates on the resolvent of $A$ were established
in \cite{MR736974}. To obtain these estimates, relatively strong
regularity conditions on the spectral measure $\mu$ were needed.
More precisely, it was assumed in \cite{MR736974} that the spectral
measure $\mu$ has the Radon-Nikodym density $m(y),\ y\in\mathbb{S}^{d-1}$,
with respect to the surface measure on $\mathbb{S}^{d-1},$ and $m(\cdot)$
is $d$-times continuously differentiable on $\mathbb{S}^{d-1}$ and
not identically $0$. Afterwards, similar perturbations of stable-like
operators were considered in \cite{MR1248747,MR3145767,MR3201992,chen2016uniqueness};
among many other things, well-posedness of the corresponding martingale
problem was obtained in \cite{MR3201992,chen2016uniqueness}. We remark
that in \cite{chen2016uniqueness}, the jump measures of the stable-like
operator don't need to have densities with respect to the Lebesgue
measure and are merely assumed to be controlled from above and below,
respectively, by two Lévy measures of non-degenerate $\alpha$-stable
processes.

The anisotropic fractional Laplacian is a special Lévy-type generator.
A general Lévy-type generator is given by
\begin{align}
Lf(x) & =\sum_{i,j=1}^{d}a_{ij}\frac{\partial^{2}}{\partial x_{i}\partial x_{j}}f(x)+b\cdot\nabla f(x)\nonumber \\
 & \qquad+\int_{\mathbb{R}^{d}\backslash\{0\}}\left[f(x+y)-f(x)-\mathbf{1}_{\{|y|\le1\}}y\cdot\nabla f(x)\right]\nu(dy),\label{MPAFLdefiofA-1}
\end{align}
where $(a_{ij})_{1\le i,j\le d}$ is a positive semi-definite symmetric
$d\times d$ matrix, $b\in\Rd$, and $\nu$ is a Lévy measure on $\mathbb{R}^{d}\backslash\{0\}$.
The tuple $\left((a_{ij})_{1\le i,j\le d},b,\nu\right)$ is called
the Lévy triple of $L$. In this paper, we study the martingale problem
for (time-dependent) non-local perturbations of a general Lévy-type
generator whose Lévy measure is controlled from below by that of a
non-degenerate anisotropic fractional Laplacian.

Our main result is  the following:
\begin{thm}
\label{thm: main}Let $L$ be as in \emph{(}\ref{MPAFLdefiofA-1}\emph{)}
and assume that there exist some $\alpha\in(0,2)$ and a non-degenerate
finite measure $\mu$ on $\mathbb{S}^{d-1}$ such that
\begin{equation}
\nu(B)\ge\int_{\mathbb{S}^{d-1}}\mu(d\xi)\int_{0}^{\infty}\mathbf{1}_{B}(r\xi)\frac{dr}{r^{1+\alpha}},\quad\forall B\in\mathcal{B}(\mathbb{R}^{d}).\label{eq:MPAFLdefiofnu-1}
\end{equation}
Define the operator $K_{t}$ by
\begin{equation}
K_{t}f(x):=\int_{\mathbb{R}^{d}\backslash\{0\}}[f(x+y)-f(x)-\mathbf{1}_{\alpha>1}\mathbf{1}_{\{|y|\le1\}}y\cdot\nabla f(x)]M(t,x,dy),\label{first defi: K_t}
\end{equation}
where $M$ is a measurable kernel from $\hs$ to $\mathcal{B}\left(\mathbb{R}^{d}\backslash\{0\}\right)$
and $M(t,x,\cdot)$ is a Lévy measure on $\mathbb{R}^{d}\backslash\{0\}$
for each $(t,x)\in\hs$. If\emph{ }there exists some $\beta\in(0,\alpha)$
such that
\begin{equation}
\sup_{t\geq0,x\in\mathbb{R}^{d}}\int_{\mathbb{R}^{d}\backslash\{0\}}1\wedge|y|^{\beta}M(t,x,dy)<\infty,\label{conditon 2 for K}
\end{equation}
then the martingale problem for $\mathcal{L}_{t}=L+K_{t}$ is well-posed.
\end{thm}
Note that the maxtrix $(a_{ij})_{1\le i,j\le d}$ in (\ref{MPAFLdefiofA-1})
is not assumed to be non-degenerate in Theorem \ref{thm: main}. Indeed,
if $(a_{ij})_{1\le i,j\le d}$ is non-degenerate, then by the classical
results of Stroock \cite{MR0433614}, the assumption (\ref{conditon 2 for K})
in Theorem \ref{thm: main} can be relaxed to $\sup_{t\geq0,x\in\mathbb{R}^{d}}\int_{\mathbb{R}^{d}\backslash\{0\}}1\wedge|y|^{2}M(t,x,dy)<\infty$.
Here we are more interested in the case where $(a_{ij})_{1\le i,j\le d}$
is degenerate and the non-local part of $L$ acts as the leading term.

The novelty of our Theorem \ref{thm: main}, compared to the results
of \cite{MR736974,MR3201992,chen2016uniqueness} in this direction,
lies firstly in the fact that the generator $L$ here contains a possibly
degenerate diffusion part. As far as the author knows, non-local perturbations
of this kind of Lévy-type generators have not yet been considered.
Another point we would like to mention is that the Lévy measure $\nu$
of $L$ is only required to satisfy the lower bound condition (\ref{conditon 2 for K}),
which is weaker than those assumed in the above mentioned works. As
a compensation, our assumption (\ref{conditon 2 for K}) on the perturbing
jump kernel $M(t,x,\cdot)$, which guarantees that $K_{t}$ is a lower
order perturbation of $L$, is actually slightly stronger than those
in \cite{MR736974,MR3201992}.

Our strategy to prove the asserted uniqueness is motivated by the
method of Komatsu in \cite{MR736974}. We will derive some non-local
estimates of the resolvent of $L$. Since our assumption on the Lévy
measure $\nu$ is much weaker than that of \cite{MR736974}, together
with the presence of the possibly degenerate diffusion part of $L$
and the time-dependency of the kernel $M(t,x,\cdot)$, our arguments
are technically more involved. To obtain the existence, we will first
consider smooth approximations $\mathcal{L}_{n,t}$ of $\mathcal{L}_{t}$
and then derive some Krylov's estimates for the martingale solutions
corresponding to $\mathcal{L}_{n,t}$. It turns out that the limit
point (under the topology of weak convergence for measures) of these
martingale solutions exists and solves the martingale problem for
$\mathcal{L}_{t}$.

The rest of the paper is organized as follows. In Section 2 we give
some notation and recall the definition of the martingale problem
for non-local generators. In Section 3 we establish some estimates
on the time-space resolvent of the Lévy process with generator $L$.
In Section 4 we construct the time-space resolvent corresponding to
$\mathcal{L}_{t}$. Finally, we prove Theorem \ref{thm: main} in
Section 5.

\section{Preliminaries}

The inner product of $x$ and $y$ in $\mathbb{R}^{d}$ is written
as $x\cdot y$. We use $|v|$ to denote the Euclidean norm of a vector
$v\in\BR^{m}$, $m\in\BN$. For a bounded function $g:\hs\to\BR^{m}$
we write $\|g\|:=\sup_{(s,x)\in\hs}|g(s,x)|$. Let $\mathbb{S}^{d-1}:=\{x\in\mathbb{R}^{d}:|x|=1\}$
be the unitary sphere.

Let $C_{b}^{2}(\mathbb{R}^{d})$ denote the class of $C^{2}$ functions
such that the function and its first and second order partial derivatives
are bounded. Note that $C_{b}^{2}(\mathbb{R}^{d})$ is a Banach space
endowed with the norm
\[
\|f\|_{C_{b}^{2}(\mathbb{R}^{d})}:=\|f\|+\sum_{i=1}^{d}\|\partial_{i}f\|+\sum_{i,j=1}^{d}\|\partial_{ij}^{2}f\|,\quad f\in C_{b}^{2}(\mathbb{R}^{d}),
\]
where $\partial_{i}f(x):=\partial_{x_{i}}f(x)$ and $\partial_{ij}^{2}f(x):=\partial_{x_{i}x_{j}}^{2}f(x)$
for $x\in\Rd$. For $k\in\BN$ and $k\ge3$, the space $C_{b}^{k}(\Rd)$
and the norm on $C_{b}^{k}(\Rd)$ are similarly defined.

Consider a Lévy-type generator
\begin{align}
Lf(x) & =\sum_{i,j=1}^{d}a_{ij}\frac{\partial^{2}}{\partial x_{i}\partial x_{j}}f(x)+b\cdot\nabla f(x)\nonumber \\
 & \qquad+\int_{\mathbb{R}^{d}\backslash\{0\}}\left[f(x+y)-f(x)-\mathbf{1}_{\{|y|\le1\}}y\cdot\nabla f(x)\right]\nu(dy),\label{MPAFLdefiofA}
\end{align}
defined for every $f\in C_{b}^{2}(\Rd)$, where $(a_{ij})_{1\le i,j\le d}$
is a positive semi-definite symmetric $d\times d$ matrix, $b\in\Rd$,
and $\nu$ is a Lévy measure on $\mathbb{R}^{d}\backslash\{0\}.$

Throughout this paper, we assume that the generator $L$ satisfies
the following assumption.
\begin{assumption}
\label{MPAFLass21}There exist $\alpha\in(0,2)$ and a non-degenerate
finite measure $\mu$ on $\mathbb{S}^{d-1}$ such that
\begin{equation}
\nu(B)\ge\int_{\mathbb{S}^{d-1}}\mu(d\xi)\int_{0}^{\infty}\mathbf{1}_{B}(r\xi)\frac{dr}{r^{1+\alpha}},\quad\forall B\in\mathcal{B}(\mathbb{R}^{d}).\label{eq:MPAFLdefiofnu}
\end{equation}
\end{assumption}
By non-degeneracy of $\mu$ we mean that the support of $\mu$ is
not contained in a proper linear subspace of $\mathbb{R}^{d}$.
\begin{rem}
Since we don't assume additional conditions on $(a_{ij})_{1\le i,j\le d}$,
the matrix $(a_{ij})_{1\le i,j\le d}$ can be degenerate.
\end{rem}
Recall that $K_{t}$ is given by
\begin{equation}
K_{t}f(x)=\int_{\mathbb{R}^{d}\backslash\{0\}}[f(x+y)-f(x)-\mathbf{1}_{\alpha>1}\mathbf{1}_{\{|y|\le1\}}y\cdot\nabla f(x)]M(t,x,dy),\label{MPAFLdefiofKt}
\end{equation}
where $M$ is a kernel from $\hs$ to $\mathcal{B}\left(\mathbb{R}^{d}\backslash\{0\}\right)$
with $M(t,x,\cdot)$ being a Lévy measure on $\mathbb{R}^{d}\backslash\{0\}$
for each $(t,x)\in\hs$. Without any further specification, we will
always assume the following:
\begin{assumption}
\label{eq:MAPAFLassonM}There exists $\beta\in(0,\alpha)$ such that
\[
\sup_{t\geq0,x\in\mathbb{R}^{d}}\int_{\mathbb{R}^{d}\backslash\{0\}}1\wedge|y|^{\beta}M(t,x,dy)<\infty.
\]

\end{assumption}
Let
\begin{equation}
\mathcal{L}_{t}:=L+K_{t},\label{MPAFLgeneratorlt}
\end{equation}
where $L$ and $K_{t}$ are defined in (\ref{MPAFLdefiofA}) and (\ref{MPAFLdefiofKt}),
respectively.

Let $D=D\big([0,\infty)\big)$, the set of paths in $\mathbb{R}^{d}$
that are right continuous with left limits, endowed with the Skorokhod
topology. Set $X_{t}(\omega)=\omega(t)$ for $\omega\in D$ and let
$\mathcal{D}=\sigma(X_{t}:0\le t<\infty)$ and $\mathcal{F}_{t}:=\sigma(X_{r}:0\le r\le t)$.
A probability measure $\mathbf{P}$ on $(D,\mathcal{D})$ is called
a solution to the martingale problem for $\mathcal{L}_{t}$ starting
from $(s,x)$, if
\begin{equation}
\mathbf{P}(X_{t}=x,\ \forall t\le s)=1\label{MPAFLwsintegrcondi}
\end{equation}
and under the measure $\mathbf{P}$,
\begin{equation}
f(X_{t})-\int_{s}^{t}\mathcal{L}_{u}f(X_{u})du,\quad t\ge s,\label{MPAFLeq2defimp}
\end{equation}
is an $\mathcal{F}_{t}$-martingale after time $s$ for all $f\in C_{b}^{2}(\mathbb{R}^{d})$.

\section{Estimates on the time-space resolvent of the Lévy process with generator
$L$}

In this section we consider a $d$-dimensional Lévy process $S=(S_{t})_{t\geq0}$
with generator $L$ that is defined in (\ref{MPAFLdefiofA}). So $S$
has the Lévy triple $((a_{ij})_{1\le i,j\le d},b,\nu)$, namely,
\begin{align}
\bfE\big[e^{iS_{t}\cdot u}\big] & =e^{-t\psi(u)},\quad u\in\Rd,\nonumber \\
\psi(u) & =\sum_{i,j=1}^{d}a_{ij}u_{i}u_{j}-\int_{\mathbb{R}^{d}\setminus\{0\}}\Big(e^{iu\cdot y}-1-\mathbf{1}_{\{|y|\le1\}}iu\cdot y\Big)\nu(dy)-ib\cdot u,\label{MPAFLeqsect21}
\end{align}
where $(a_{ij})_{1\le i,j\le d}$, $b$ and $\nu$ are the same as
in (\ref{MPAFLdefiofA}).

Let $\alpha\in(0,2)$ and $\mu$ be as in Assumption \ref{MPAFLass21}.
Define
\begin{equation}
\tilde{\nu}(B)=\int_{\mathbb{S}^{d-1}}\mu(d\xi)\int_{0}^{\infty}\mathbf{1}_{B}(r\xi)\frac{dr}{r^{1+\alpha}},\quad B\in\mathcal{B}(\mathbb{R}^{d}),\label{defi: nu tilde}
\end{equation}
and
\begin{equation}
\tilde{\psi}(u)=-\int_{\mathbb{R}^{d}\setminus\{0\}}\Big(e^{iu\cdot y}-1-\mathbf{1}_{\{|y|\le1\}}iu\cdot y\Big)\tilde{\nu}(dy),\quad u\in\Rd.\label{defi: psi tilde}
\end{equation}
Then $\tilde{\psi}$ is the characteristic exponent of an $\alpha$-stable
process $\tilde{S}=(\tilde{S}_{t})_{t\geq0}$. Let $\hat{\psi}:=\psi-\tilde{\psi}$.
So $\hat{\psi}$ is the characteristic exponent of a Lévy process
$\hat{S}=(\hat{S}_{t})_{t\geq0}$ with the Lévy triple $(A,b,\nu-\tilde{\nu})$.
Without loss of generality, we assume that $S$, $\tilde{S}$ and
$\hat{S}$ are defined on the same probability space.

Define
\begin{equation}
\gamma:=\begin{cases}
-\int_{\{0<|y|\le1\}}y\tilde{\nu}(dy), & 0<\alpha<1,\\
\int_{\mathbb{S}^{d-1}}\xi\mu(d\xi), & \alpha=1,\\
\int_{\{|y|>1\}}y\tilde{\nu}(dy), & 1<\alpha<2.
\end{cases}\label{defi, gamma}
\end{equation}
Then for $\alpha\neq1$, the function $\tilde{\psi}(u)+iu\cdot\gamma$
becomes a homogeneous function (with variable $u$) of index $\alpha$.
As a result, for $\alpha\neq1$, we obtain
\begin{equation}
\tilde{\psi}(\rho u)+i(\rho u\cdot\gamma)=\rho^{\alpha}(\tilde{\psi}(u)+i(u\cdot\gamma)),\quad\forall\rho>0.\label{MPAFLeqsect215-1}
\end{equation}
The case with $\alpha=1$ is a little different. For $\alpha=1$,
according to \cite[p.~84,~(14.20)]{MR1739520} and its complex conjugate,
it holds that

\[
\tilde{\psi}(u)=\int_{\mathbb{S}^{d-1}}\left(\frac{\pi}{2}|u\cdot\xi|+iu\cdot(\xi\log|u\cdot\xi|)-ic_{1}u\cdot\xi\right)\mu(d\xi),\quad u\in\mathbb{R}^{d},
\]
where $c_{1}=\int_{1}^{\infty}r^{2}\sin rdr+\int_{0}^{1}r^{-2}(\sin r-r)dr$;
in this case, we have
\begin{equation}
\tilde{\psi}(\rho u)=\rho\tilde{\psi}(u)+i(\rho\log\rho)u\cdot\gamma,\quad\forall\rho>0,\ u\in\mathbb{R}^{d}.\label{MPAFLeqsect217}
\end{equation}

According to Assumption \ref{MPAFLass21} and \cite[Prop.~24.20]{MR1739520},
there exists some constant $c_{2}>0$ such that
\begin{equation}
\left|e^{-t\tilde{\psi}(u)}\right|\le e^{-c_{2}t|u|^{\alpha}},\quad\forall u\in\Rd,\ t>0.\label{MPAFLsect31}
\end{equation}
By the inversion formula of Fourier transform, the law of $\tilde{S}_{t}$
has a density $\tilde{p}_{t}\in L^{1}(\mathbb{R}^{d})\cap C_{b}(\mathbb{R}^{d})$
that is given by
\begin{equation}
\tilde{p}_{t}(x)=\frac{1}{(2\pi)^{d}}\int_{\mathbb{R}^{d}}e^{-iu\cdot x}e^{-t\tilde{\psi}(u)}du,\quad x\in\mathbb{R}^{d},\ t>0.\label{defi: p tilde}
\end{equation}
Moreover, according to \cite[p.~2856,~(2.3)]{MR2286060}, we have
the following scaling property for $\tilde{p}_{t}$: for $x\in\mathbb{R}^{d},\ t>0,$
\begin{equation}
\tilde{p}_{t}(x)=\begin{cases}
t^{-d/\alpha}\tilde{p}_{1}(t^{-1/\alpha}x+(1-t^{1-1/\alpha})\gamma),\quad & (\alpha\neq1),\\
t^{-d}\tilde{p}_{1}(t^{-1}x-\gamma\log t), & (\alpha=1),
\end{cases}\label{MPAFLsect315}
\end{equation}
 where $\gamma$ is given in (\ref{defi, gamma}).

The following result is a slight extension of \cite[Lemma 3.1]{MR2945756}.
For its proof the reader is referred to \cite[Lemma 3.1]{jin2015weak}.
\begin{lem}
\label{lemma: p_t tilde}Let $t>0$ be arbitrary. Then the densities
$\tilde{p}_{t}\in C_{b}^{\infty}(\Rd)\cap L^{r}(\Rd)$ for all $r\ge1$.
\end{lem}
Since
\[
\bfE\big[e^{iS_{t}\cdot u}\big]=e^{-t\psi(u)}=e^{-t\tilde{\psi}(u)}e^{-t\hat{\psi}(u)}=\bfE\big[e^{i\tilde{S}_{t}\cdot u}\big]\bfE\big[e^{i\hat{S}_{t}\cdot u}\big],
\]
the law of $S_{t}$ has a density $p_{t}$ that is given by
\begin{equation}
p_{t}(x):=\int_{\Rd}\tilde{p}_{t}(x-y)\hat{m}_{t}(dy),\quad x\in\mathbb{R}^{d},\ t>0,\label{defi: p}
\end{equation}
where $\hat{m}_{t}$ denotes the law of $\hat{S}_{t}$. It follows
from Lemma \ref{lemma: p_t tilde} that $p_{t}\in C_{b}^{\infty}(\Rd)\cap L^{r}(\Rd)$
for all $r\ge1$.

For $0<\delta<1$, define the integro-differential operator $|\partial|^{\delta}$
by

\[
|\partial|^{\delta}f(x)=c_{3}\int_{\mathbb{R}^{d}\backslash\{0\}}\left[f(x+y)-f(x)\right]\cdot|y|^{-d-\delta}dy,\quad f\in C_{b}^{2}(\Rd),
\]
where the constant $c_{3}$ is given by
\[
c_{3}:=2^{\delta}\pi^{-d/2}\Gamma\left(\frac{d+\delta}{2}\right)/\Gamma\left(-\frac{\delta}{2}\right).
\]
Note that
\begin{equation}
c_{3}\int_{\mathbb{R}^{d}\setminus\{0\}}\Big(e^{iu\cdot y}-1\Big)|y|^{-d-\delta}dy=-|u|^{\delta},\quad u\in\mathbb{R}^{d}.\label{eq: u^delta}
\end{equation}

Next, we give an estimate of the $L^{r}$-norm of $|\partial|^{\delta}p_{t}$.
\begin{lem}
\label{lem: lp esti for delta p_t}Let $0<\delta<1$ and $r\ge1$.
Then there exists a constant $c_{4}>0$ that depends on $\delta$
and $r$ such that
\begin{equation}
\||\partial|^{\delta}p_{t}\|_{L^{r}(\Rd)}\le c_{4}t^{(d/r-\delta-d)/\alpha},\quad\forall t>0.\label{eq:MPAFLgtxl1norm}
\end{equation}
\end{lem}
\begin{proof}
Since $|\partial|^{\delta}p_{t}(x)=\int_{\Rd}|\partial|^{\delta}\tilde{p}_{t}(x-y)\hat{m}_{t}(dy)$,
$t>0$, by Jensen's inequality, it suffices to prove
\[
\||\partial|^{\delta}\tilde{p}_{t}\|_{L^{r}(\Rd)}\le t^{(d/r-\delta-d)/\alpha}\||\partial|^{\delta}\tilde{p}_{1}\|_{L^{r}(\Rd)}<\infty,\quad\forall t>0.
\]
By (\ref{MPAFLsect31}), (\ref{defi: p tilde}) and Fubini's theorem,
we easily obtain that for each $t>0$,

\begin{equation}
|\partial|^{\delta}\tilde{p}_{t}(x)=-\frac{1}{(2\pi)^{d}}\int_{\mathbb{R}^{d}}|u|^{\delta}e^{-t\tilde{\psi}(u)}e^{-iu\cdot x}du,\quad x\in\mathbb{R}^{d}.\label{eq:MPAFLsect316}
\end{equation}

We first assume $\alpha\neq1$. Using a change of variables $u=t^{-1/\alpha}u'$
and noting (\ref{MPAFLeqsect215-1}), we obtain
\begin{align*}
|\partial|^{\delta}\tilde{p}_{t}(x)= & -\frac{t^{-d/\alpha}}{(2\pi)^{d}}\int_{\Rd}t^{-\delta/\alpha}|u'|^{\delta}e^{-(\tilde{\psi}(u')+iu'\cdot\gamma)+it^{1-1/\alpha}u'\cdot\gamma}e^{-it^{-1/\alpha}u'\cdot x}du'\\
= & t^{-(\delta+d)/\alpha}|\partial|^{\delta}\tilde{p}_{1}\left(t^{-1/\alpha}x-\gamma(t^{1-1/\alpha}-1)\right).
\end{align*}
So
\begin{align}
\||\partial|^{\delta}\tilde{p}_{t}\|_{L^{r}(\Rd)} & \le t^{-(\delta+d)/\alpha}\left(\int_{\Rd}\left(|\partial|^{\delta}\tilde{p}_{1}(t^{-1/\alpha}x)\right)^{r}dx\right)^{1/r}\nonumber \\
 & =t^{(d/r-\delta-d)/\alpha}\||\partial|^{\delta}\tilde{p}_{1}\|_{L^{r}(\Rd)}.\label{MPAFLeqremark3156}
\end{align}
For the case $\alpha=1$, we can apply (\ref{MPAFLeqsect217}) and
a similar argument as above to also obtain (\ref{MPAFLeqremark3156}).
So (\ref{MPAFLeqremark3156}) is true for all $\alpha\in(0,2)$.

It remains to show that $\||\partial|^{\delta}\tilde{p}_{1}\|_{L^{r}(\Rd)}<\infty$,
or equivalently,
\begin{equation}
\int_{\Rd}\Big|\int_{\Rd}|u|^{\delta}e^{-\tilde{\psi}(u)}e^{-iu\cdot y}du\Big|^{r}dy<\infty.\label{eq2:MPAFLlemma6udel}
\end{equation}
To prove this fact, we use the same idea as in the proof of \cite[Lemma~3.4]{jin2015weak}.
Firstly, note that the characteristic exponent $\tilde{\psi}$ can
be written as the sum of $\tilde{\psi}_{1}$ and $\tilde{\psi}_{2}$,
where
\[
\tilde{\psi}_{1}(u)=-\int_{\{0<|y|\le1\}}\Big(e^{iu\cdot y}-1-iu\cdot y\Big)\tilde{\nu}(dy),\quad\tilde{\psi}_{2}=\tilde{\psi}-\tilde{\psi}_{1}.
\]
We can easily check that that $\tilde{\psi}_{1}\in C^{\infty}(\mathbb{R}^{d})$.
Since (\ref{MPAFLsect31}) holds, we see that $\exp(-\tilde{\psi}_{1})$
belongs to the Schwartz space $\mathcal{S}(\mathbb{R}^{d})$.

According to (\ref{eq: u^delta}), we can write $|u|^{\delta}=\psi_{\delta,1}(u)+\psi_{\delta,2}(u)+\psi_{\delta,3}$,
where
\[
\psi_{\delta,1}(u)=-c_{\delta}\int_{\{0<|y|\le1\}}\left(e^{iu\cdot y}-1\right)|y|^{-d-\delta}dy
\]
and
\[
\psi_{\delta,2}(u)=-c_{\delta}\int_{\{|y|>1\}}e^{iu\cdot y}|y|^{-d-\delta}dy,\quad\psi_{\delta,3}=c_{\delta}\int_{\{|y|>1\}}|y|^{-d-\delta}dy.
\]
Then
\begin{align}
|u|^{\delta}e^{-\tilde{\psi}}= & \psi_{\delta,1}e^{-\tilde{\psi}_{1}}e^{-\tilde{\psi}_{2}}+\psi_{\delta,2}e^{-\tilde{\psi}_{1}}e^{-\tilde{\psi}_{2}}+\psi_{\delta,3}e^{-\tilde{\psi}_{1}}e^{-\tilde{\psi}_{2}}\nonumber \\
= & \psi_{\delta,1}e^{-\tilde{\psi}_{1}}e^{-\tilde{\psi}_{2}}-e^{-\tilde{\psi}_{1}}(-\psi_{\delta,2})e^{-\tilde{\psi}_{2}}+\psi_{\delta,3}e^{-\tilde{\psi}_{1}}e^{-\tilde{\psi}_{2}}.\label{eq1:MPAFLlemma6udel}
\end{align}

We only treat the first term on the right-hand side of (\ref{eq1:MPAFLlemma6udel}),
since the other two terms are similar. With the same reason as for
$\exp(-\tilde{\psi}_{1})$ above, we have $\psi_{\delta,1}\exp(-\tilde{\psi}_{1})\in\mathcal{S}(\mathbb{R}^{d})$.
It is also easy to see that $\exp(-\tilde{\psi}_{2})$ is bounded
and is the characteristic function of an infinitely divisible probability
measure $\rho$ on $\mathbb{R}^{d}$. As a consequence, we are allowed
to define $h$ to be the inverse Fourier transform of the $\psi_{\delta,1}\exp(-\tilde{\psi})$,
i.e.,
\[
h(y):=\frac{1}{(2\pi)^{d}}\int_{\Rd}\psi_{\delta,1}e^{-\tilde{\psi}_{1}}e^{-\tilde{\psi}_{2}}e^{-iu\cdot y}du,\quad y\in\Rd.
\]
Since the Fourier transform is a one-to-one map of $\mathcal{S}(\mathbb{R}^{d})$
onto itself, we can find $f\in\mathcal{S}(\mathbb{R}^{d})$ with $\hat{f}$=$\psi_{\delta,1}\exp(-\tilde{\psi}_{1})$,
where $\hat{f}$ denotes the Fourier transform of $f$. In particular,
we have $f\in L^{r}(\mathbb{R}^{d})$. Let $f*\rho$ be the convolution
of $f$ and $\rho.$ We have
\[
\ \widehat{f*\rho}=\hat{f}\hat{\rho}=\psi_{\delta,1}e^{-\tilde{\psi}_{1}-\tilde{\psi}_{2}}=\psi_{\delta,1}e^{-\tilde{\psi}}=\hat{h},
\]
which implies $h=f*\rho$. Thus $h\in C_{b}^{\infty}(\mathbb{R}^{d})$.
By Young's inequality, we get $h\in L^{r}(\mathbb{R}^{d})$, i.e.,
\begin{equation}
\int_{\Rd}\Big|\int_{\Rd}\psi_{\delta,1}e^{-\tilde{\psi}(u)}e^{-iu\cdot y}du\Big|^{r}dy<\infty.\label{eq3:MPAFLlemma6udel}
\end{equation}
Similarly, by noting that $-\psi_{\delta,2}$ and $\exp(-\tilde{\psi}_{2})$
are both characteristic functions of some finite measures on $\Rd$,
we can show that
\begin{equation}
\int_{\Rd}\Big|\int_{\Rd}e^{-\tilde{\psi}_{1}}(-\psi_{\delta,2})e^{-\tilde{\psi}_{2}}e^{-iu\cdot y}du\Big|^{r}dy<\infty\label{eq4:MPAFLlemma6udel}
\end{equation}
and
\begin{equation}
\int_{\Rd}\Big|\int_{\Rd}\psi_{\delta,3}e^{-\tilde{\psi}_{1}}e^{-\tilde{\psi}_{2}}e^{-iu\cdot y}du\Big|^{r}dy<\infty.\label{eq5:MPAFLlemma6udel}
\end{equation}
Now, the inequality (\ref{eq2:MPAFLlemma6udel}) follows from (\ref{eq1:MPAFLlemma6udel}),
(\ref{eq3:MPAFLlemma6udel}), (\ref{eq4:MPAFLlemma6udel}) and (\ref{eq5:MPAFLlemma6udel}).
\end{proof}
\begin{rem}
If we understand $|\delta|^{0}$ as the identity map, then Lemma \ref{lem: lp esti for delta p_t}
holds also for the case $\delta=0$, namely, for each $r\ge1$, there
exists a constant $c_{4}>0$ depending on $r$ such that
\begin{equation}
\|p_{t}\|_{L^{r}(\Rd)}\le c_{4}t^{(d/r-d)/\alpha},\quad\forall t>0.\label{esti: lp of p_t}
\end{equation}
Indeed, the proof of Lemma \ref{lem: lp esti for delta p_t} can be
easily adapted to work also for this case.
\end{rem}
In the next lemma we deal with a non-local estimate on the gradient
of $p_{t}$ when $1<\alpha<2$. Since its proof is completely similar
to that of Lemma \ref{lem: lp esti for delta p_t}, so we omit it
here.
\begin{lem}
\label{lem: detla i p_t}Let $1<\alpha<2$ , $0<\delta<\alpha-1$
and $r\ge1$. Then there exists a constant $c_{5}>0$ which depends
on $\delta$ and $r$ such that for each $i=1,\cdots,d$,
\[
\||\partial|^{\delta}\partial_{i}p_{t}\|_{L^{r}(\Rd)}\le c_{5}t^{(d/r-\delta-1-d)/\alpha},\quad\forall t>0.
\]
\end{lem}
For $\lambda>0$, the time-space resolvent $R_{\lambda}$ of the Lévy
process $S$ is defined by
\begin{equation}
R_{\lambda}f(t,x):=\int_{0}^{\infty}e^{-\lambda u}\int_{\mathbb{R}^{d}}p_{u}(y-x)f(t+u,y)dydu,\quad(t,x)\in\hs,\label{eq:MPAFLdefiofrlam}
\end{equation}
where $f\in\mathcal{B}_{b}(\mathbb{R}_{+}\times\mathbb{R}^{d})$.

Before we state the next lemma, we recall two equalities from \cite[Lemma 2.1]{MR736974}:
for each $0<\delta<1,$ there exist konstants $c_{6},$$c_{7}>0$,
which depend on $\delta$, such that
\begin{equation}
\int_{\mathbb{R}^{d}}\left|(|w+z|^{\delta-d}-|w|^{\delta-d})\right|dw=c_{6}|z|^{\delta},\label{eq1:Komastu}
\end{equation}
and
\begin{equation}
f(x+z)-f(x)=c_{7}\int_{\mathbb{R}^{d}}(|w+z|^{\delta-d}-|w|^{\delta-d})|\partial|^{\delta}f(x-w)dw,\label{eq2:Komastu}
\end{equation}
where $f\in C_{b}^{\infty}(\mathbb{R}^{d})$ is arbitrary.
\begin{lem}
Assume $0<\delta<\alpha\wedge1.$

\emph{(i)} If $\lambda>0$ and $g\in\mathcal{B}_{b}(\mathbb{R}_{+}\times\mathbb{R}^{d})$,
then $|\partial|^{\delta}\left(R_{\lambda}g(t,\cdot)\right)$ is well-defined
for each $t\ge0$. Moreover, there exists a constant $C_{\lambda}>0$,
independent of $g$, such that
\begin{equation}
\left||\partial|^{\delta}\left(R_{\lambda}g(t,\cdot)\right)(x)\right|\le C_{\lambda}\|g\|\label{new new eq0}
\end{equation}
and
\begin{equation}
|R_{\lambda}g(t,x+z)-R_{\lambda}g(t,x)|\le C_{\lambda}|z|^{\delta}\|g\|\label{new new eq 0.5}
\end{equation}
for all $(t,x)\in\hs$ and $z\in\Rd$. The constant $C_{\lambda}$
goes to $0$ as $\lambda\to\infty$.

\emph{(ii)} Let $T>0$ and $g\in\mathcal{B}_{b}(\mathbb{R}_{+}\times\mathbb{R}^{d})$
be such that $supp(g)\subset[0,T]\times\Rd$ and $g\in L^{q}([0,T];L^{p}(\Rd))$
with $p,q>0$ and $d/p+\alpha/q<\alpha-\delta$. Then for each $\lambda>0$,
there exists a constant $N_{\lambda}>0$, independent of $g$ and
$T$, such that
\begin{equation}
\left||\partial|^{\delta}\left(R_{\lambda}g(t,\cdot)\right)(x)\right|\le N_{\lambda}\|g\|_{L^{q}([0,T];L^{p}(\Rd))}\label{neweq 2: Lemma 3.4}
\end{equation}
and
\begin{equation}
|R_{\lambda}g(t,x+z)-R_{\lambda}g(t,x)|\le N_{\lambda}z^{\delta}\|g\|_{L^{q}([0,T];L^{p}(\Rd))}\label{neweq3: Lemma 3.4}
\end{equation}
for all $(t,x)\in\hs$ and $z\in\Rd$. Moreover, the constant $N_{\lambda}$
goes to $0$ as $\lambda\to\infty$.\label{lem4:MPAFL}
\end{lem}
\begin{proof}
(i) Assume $g\in\mathcal{B}_{b}(\mathbb{R}_{+}\times\mathbb{R}^{d})$.
Let $\epsilon>0$ be a constant such that $\delta<\delta+\epsilon<\alpha\wedge1$.
For $z\in\Rd$, we have
\begin{align}
 & |p_{u}(y-x-z)-p_{u}(y-x)|\nonumber \\
\overset{(\ref{eq2:Komastu})}{\le} & c_{7}\int_{\mathbb{R}^{d}}\Big|(|w-z|^{\delta+\epsilon-d}-|w|^{\delta+\epsilon-d})|\partial|^{\delta+\epsilon}p_{u}(y-x-w)\Big|dw.\label{eq3:MPFALdifferofpt}
\end{align}
It follows from (\ref{eq3:MPFALdifferofpt}) and Young's inequality
that
\begin{align}
 & \int_{\mathbb{R}^{d}}|p_{u}(y-x-z)-p_{u}(y-x)|dy\nonumber \\
\le & c_{7}\||\partial|^{\delta+\epsilon}p_{u}\|_{L^{1}(\Rd)}\int_{\mathbb{R}^{d}}|(|w-z|^{\delta+\epsilon-d}-|w|^{\delta+\epsilon-d})|dw\nonumber \\
\overset{(\ref{eq1:Komastu})}{=} & c_{6}c_{7}|z|^{\delta+\epsilon}\||\partial|^{\delta+\epsilon}p_{u}\|_{L^{1}(\Rd)}\overset{(\ref{eq:MPAFLgtxl1norm})}{\le}c_{4}c_{6}c_{7}u^{-(\delta+\epsilon)/\alpha}|z|^{\delta+\epsilon}.\label{ineq:MPAFLformulaI2}
\end{align}
So
\begin{align}
\left|R_{\lambda}g(t,x+z)-R_{\lambda}g(t,x)\right| & \le\|g\|\int_{0}^{\infty}e^{-\lambda u}\int_{\mathbb{R}^{d}}|p_{u}(y-x-z)-p_{u}(y-x)|dydu\nonumber \\
 & \le c_{4}c_{6}c_{7}|z|^{\delta+\epsilon}\|g\|\int_{0}^{\infty}e^{-\lambda u}u^{-(\delta+\epsilon)/\alpha}du.\label{esti 1: R lambda}
\end{align}
On the other hand, we have
\begin{equation}
\left|R_{\lambda}g(t,x+z)-R_{\lambda}g(t,x)\right|\le2\|R_{\lambda}g\|\le2\lambda^{-1}\|g\|.\label{new new eq 2}
\end{equation}
By (\ref{esti 1: R lambda}) and (\ref{new new eq 2}), we can find
a constant $c>0$ such that
\[
\left|R_{\lambda}g(t,x+z)-R_{\lambda}g(t,x)\right|\le c\left(|z|^{\delta+\epsilon}\wedge1\right),\quad\forall z\in\Rd,
\]
which implies that $|\partial|^{\delta}\left(R_{\lambda}g(t,\cdot)\right)(x)$
is well-defined.

By Fubini's theorem, we obtain that for all $t\ge0$ and $x\in\Rd$,
\begin{equation}
|\partial|^{\delta}\left(R_{\lambda}g(t,\cdot)\right)(x)=\int_{0}^{\infty}e^{-\lambda u}\int_{\mathbb{R}^{d}}|\partial|^{\delta}\left(p_{u}(y-\cdot)\right)(x)g(t+u,y)dydu.\label{eq:MPAFLrepreforpartialdelta}
\end{equation}
So for all $t\ge0$ and $x\in\Rd$,
\begin{equation}
\left||\partial|^{\delta}\left(R_{\lambda}g(t,\cdot)\right)(x)\right|\le\|g\|\int_{0}^{\infty}e^{-\lambda u}\left\Vert |\partial|^{\delta}p_{u}\right\Vert _{L^{1}(\Rd)}du\overset{(\ref{eq:MPAFLgtxl1norm})}{\le}C_{\lambda}\|g\|,\label{neweq: Lemma 3.4}
\end{equation}
where
\[
C_{\lambda}:=c_{4}\int_{0}^{\infty}e^{-\lambda u}u^{-\delta/\alpha}du.
\]
Hence (\ref{new new eq0}) is true. It is clear that $C_{\lambda}\downarrow0$
as $\lambda\to\infty$.

It follows from (\ref{eq2:Komastu}) that
\begin{align*}
 & R_{\lambda}g(t,x+z)-R_{\lambda}g(t,x)\\
 & \quad=\int_{0}^{\infty}e^{-\lambda u}\int_{\mathbb{R}^{d}}[p_{u}(y-x-z)-p_{u}(y-x)]g(t+u,y)dydu\\
 & \quad=\int_{0}^{\infty}e^{-\lambda u}\int_{\mathbb{R}^{d}}\Big(c_{7}\int_{\mathbb{R}^{d}}(|w-z|^{\delta-d}-|w|^{\delta-d})|\partial|^{\delta}p_{u}(y-x-w)dw\Big)\\
 & \qquad\quad\times g(t+u,y)dydu.
\end{align*}
In view of (\ref{eq:MPAFLgtxl1norm}), (\ref{eq1:Komastu}) and (\ref{eq:MPAFLrepreforpartialdelta}),
we can apply Fubini's theorem to obtain that for all $t\ge0,\ x,z\in\Rd$,

\begin{equation}
R_{\lambda}g(t,x+z)-R_{\lambda}g(t,x)=c_{7}\int_{\mathbb{R}^{d}}(|w-z|^{\delta-d}-|w|^{\delta-d})|\partial|^{\delta}\left(R_{\lambda}g(t,\cdot)\right)(x-w)dw.\label{eq:eq:MPAFLreprefordifferRlambda}
\end{equation}

Combining (\ref{eq:eq:MPAFLreprefordifferRlambda}), (\ref{eq1:Komastu})
and (\ref{neweq: Lemma 3.4}) yields (\ref{new new eq 0.5}).

(ii) Since (\ref{neweq3: Lemma 3.4}) follows easily from (\ref{eq1:Komastu}),
(\ref{neweq 2: Lemma 3.4}) and (\ref{eq:eq:MPAFLreprefordifferRlambda}),
we only need to prove (\ref{neweq 2: Lemma 3.4}). Note that $supp(g)\subset[0,T]\times\Rd$.
By (\ref{eq:MPAFLrepreforpartialdelta}) and Hölder's inequality,
we get
\begin{align*}
||\partial|^{\delta}\left(R_{\lambda}g(t,\cdot)\right)(x)|= & \Big|\int_{0}^{\infty}e^{-\lambda u}\int_{\mathbb{R}^{d}}|\partial|^{\delta}\left(p_{u}(y-\cdot)\right)(x)g(t+u,y)dydu\Big|\\
\le & \int_{0}^{\infty}e^{-\lambda u}\||\partial|^{\delta}p_{u}\|_{L^{p^{*}}(\Rd)}\|g(t+u,\cdot)\|_{L^{p}(\Rd)}du\\
= & \int_{0}^{T}e^{-\lambda u}\||\partial|^{\delta}p_{u}\|_{L^{p^{*}}(\Rd)}\|g(t+u,\cdot)\|_{L^{p}(\Rd)}du\\
\le & \Big(\int_{0}^{T}e^{-q^{*}\lambda u}\||\partial|^{\delta}p_{u}\|_{L^{p^{*}}(\Rd)}^{q^{*}}dt\Big)^{1/q^{*}}\|g\|_{L^{q}([0,T];L^{p}(\Rd))},
\end{align*}
where $p^{*},q^{*}>0$ are such that $1/p^{*}+1/p=1$ and $1/q^{*}+1/q=1$.
By (\ref{eq:MPAFLgtxl1norm}), we see that the inequality (\ref{neweq 2: Lemma 3.4})
holds with
\[
N_{\lambda}:=\Big(c_{4}\int_{0}^{\infty}e^{-q^{*}\lambda u}u^{q^{*}\alpha^{-1}(d/p^{*}-\delta-d)}du\Big)^{1/q^{*}},
\]
which is finite if $q^{*}\alpha^{-1}(d/p^{*}-\delta-d)>-1$, or equivalently,
$d/p+\alpha/q<\alpha-\delta$. By dominated convergence theorem, $\lim_{\lambda\to\infty}N_{\lambda}=0$.
\end{proof}
\begin{lem}
Let $1<\alpha<2$ and $0<\delta<\alpha-1$.

\emph{(i)} If $\lambda>0$ and $g\in\mathcal{B}_{b}(\mathbb{R}_{+}\times\mathbb{R}^{d})$,
then $|\partial|^{\delta}\left(\partial_{i}R_{\lambda}g(t,\cdot)\right)$
is well-defined for each $t\ge0$. Moreover, there exists a constant
$\tilde{C}_{\lambda}>0$, independent of $g$, such that for all $g\in\mathcal{B}_{b}(\mathbb{R}_{+}\times\mathbb{R}^{d})$,
\[
\left||\partial|^{\delta}\left(\partial_{i}R_{\lambda}g(t,\cdot)\right)(x)\right|\le\tilde{C}_{\lambda}\|g\|
\]
and
\[
|\partial_{i}R_{\lambda}g(t,x+z)-\partial_{i}R_{\lambda}g(t,x)|\le\tilde{C}_{\lambda}|z|^{\delta}\|g\|
\]
for all $(t,x)\in\hs$, $z\in\Rd$ and $i=1,\cdots,d$. The constant
$\tilde{C}_{\lambda}$ goes to $0$ as $\lambda\to\infty$.

\emph{(ii)} Let $T>0$ and $g\in\mathcal{B}_{b}(\mathbb{R}_{+}\times\mathbb{R}^{d})$
be such that $supp(g)\subset[0,T]\times\Rd$ and $g\in L^{q}([0,T];L^{p}(\Rd))$
with $d/p+\alpha/q<\alpha-1-\delta$. Then for each $\lambda>0$,
there exists a constant $\tilde{N}_{\lambda}>0$, independent of $g$
and $T$, such that
\[
\left||\partial|^{\delta}\left(\partial_{i}R_{\lambda}g(t,\cdot)\right)(x)\right|\le\tilde{N}_{\lambda}\|g\|_{L^{q}([0,T];L^{p}(\Rd))}
\]
and
\[
|\partial_{i}R_{\lambda}g(t,x+z)-\partial_{i}R_{\lambda}g(t,x)|\le\tilde{N}_{\lambda}z^{\delta}\|g\|_{L^{q}([0,T];L^{p}(\Rd))}.
\]
for all $(t,x)\in\hs$, $z\in\Rd$ and $i=1,\cdots,d$. Moreover,
the constant $\tilde{N}_{\lambda}$ goes to $0$ as $\lambda\to\infty$.\label{lem:MPAFLdifferrlam}
\end{lem}
\begin{proof}
Let $g\in\mathcal{B}_{b}(\mathbb{R}_{+}\times\mathbb{R}^{d})$ be
arbitrary. It is easy to see that for each $i=1,\cdots,d,$
\[
\partial_{i}R_{\lambda}g(t,x)=-\int_{0}^{\infty}e^{-\lambda u}\int_{\mathbb{R}^{d}}\partial_{i}p_{u}(y-x)g(t+u,y)dydu.
\]
In view of Lemma \ref{lem: detla i p_t}, we can argue in the same
way as in Lemma \ref{lem4:MPAFL} to derive the statements. We omit
the details.
\end{proof}

\section{Construction of the time-space resolvent corresponding to $\mathcal{L}_{t}$}

In this section we give a purely analytical construction of the time-space
resolvent $G_{\lambda}$ that corresponds to the generator $\mathcal{L}_{t}:=L+K_{t}$.
Not to be precise, we can write $G_{\lambda}=(\lambda-\partial_{t}-\mathcal{L}_{t})^{-1}$.
The main aim of this section is to establish rigorously, at least
for large enough $\lambda>0$, that
\[
G_{\lambda}g=\sum_{k=0}^{\infty}R_{\lambda}(KR_{\lambda})^{k}g,\quad g\in\mathcal{B}_{b}(\mathbb{R}_{+}\times\mathbb{R}^{d}),
\]
where $R_{\lambda}$ is the time-space resolvent of the Lévy process
$S$ and the operator $KR_{\lambda}$ is defined by
\begin{align}
KR_{\lambda}g(t,x) & :=\int_{\mathbb{R}^{d}\backslash\{0\}}[R_{\lambda}g(t,x+z)-R_{\lambda}g(t,x)\nonumber \\
 & \qquad-\mathbf{1}_{\alpha>1}\mathbf{1}_{\{|z|\le1\}}z\cdot\nabla R_{\lambda}g(t,x)]M(t,x,dz),\quad(t,x)\in\hs.\label{defi: KR_lambda}
\end{align}
To see that $KR_{\lambda}g$ in (\ref{defi: KR_lambda}) is well-defined
for $g\in\mathcal{B}_{b}(\mathbb{R}_{+}\times\mathbb{R}^{d})$, we
need the following proposition.
\begin{prop}
\label{prop MPAFL:For-any-}For each $\lambda>0,$ define
\begin{equation}
k_{\lambda}:=\begin{cases}
(C_{\lambda}+2\lambda^{-1})\sup_{t\geq0,x\in\mathbb{R}^{d}}\int_{\mathbb{R}^{d}\backslash\{0\}}1\wedge|z|^{\beta}M(t,x,dz), & 0<\alpha\leq1,\\
(\tilde{C}_{\lambda}+2\lambda^{-1})\sup_{t\geq0,x\in\mathbb{R}^{d}}\int_{\mathbb{R}^{d}\backslash\{0\}}1\wedge|z|^{\beta}M(t,x,dz), & 1<\alpha<2,
\end{cases}\label{defi: k_lambda}
\end{equation}
where $C_{\lambda}$ and $\tilde{C}_{\lambda}$ are the constants
from Lemma \ref{lem4:MPAFL} and Lemma \ref{lem:MPAFLdifferrlam},
respectively. Then
\begin{equation}
\|KR_{\lambda}g\|\leq k_{\lambda}\|g\|,\quad\forall g\in\mathcal{B}_{b}(\mathbb{R}_{+}\times\mathbb{R}^{d}).\label{esti: kr_lambda}
\end{equation}
\end{prop}
\begin{proof}
Let $\beta\in(0,\alpha)$ be the constant in Assumption \ref{eq:MAPAFLassonM}.
We distinguish between the cases with $0<\alpha\leq1$ and $1<\alpha<2$.

\emph{Case} 1: $0<\alpha\leq1$. According to Lemma \ref{lem4:MPAFL},
there exists a constant $C_{\lambda}>0$ such that for all $g\in\mathcal{B}_{b}(\mathbb{R}_{+}\times\mathbb{R}^{d})$,
\[
|R_{\lambda}g(t,x+z)-R_{\lambda}g(t,x)|\leq C_{\lambda}\|g\||z|^{\beta},\quad(t,x)\in\hs,\ z\in\Rd,
\]
and $C_{\lambda}$ goes to $0$ as $\lambda\uparrow\infty$.

Let $g\in\mathcal{B}_{b}(\mathbb{R}_{+}\times\mathbb{R}^{d})$ be
arbitrary. Then

\begin{align}
 & \int_{\mathbb{R}^{d}\backslash\{0\}}\left|R_{\lambda}g(t,x+z)-R_{\lambda}g(t,x)\right|M(t,x,dz)\nonumber \\
 & \quad=\int_{\{0<|z|\leq1\}}|R_{\lambda}g(t,x+z)-R_{\lambda}g(t,x)|M(t,x,dz)\nonumber \\
 & \qquad+{\displaystyle \int_{\{|z|>1\}}|R_{\lambda}g(t,x+z)-R_{\lambda}g(t,x)|M(t,x,dz)}\label{eq 1: Prop. 4.1}\\
{\displaystyle } & \quad\le C_{\lambda}\|g\|\int_{\{0<|z|\leq1\}}|z|^{\beta}M(t,x,dz)+2\Vert R_{\lambda}g\Vert\int_{\{|z|>1\}}1M(t,x,dz)\nonumber \\
{\displaystyle } & \quad\le(C_{\lambda}+2\lambda^{-1})\|g\|\sup_{t\geq0,x\in\mathbb{R}^{d}}\int_{\mathbb{R}^{d}\backslash\{0\}}1\wedge|z|^{\beta}M(t,x,dz).\nonumber
\end{align}
So $KR_{\lambda}g$ is well-defined and $||KR_{\lambda}g||\leq k_{\lambda}||g||$.

\emph{Case} 2: $1<\alpha<2$. Let $\delta\in(0,1)$ be such that $\beta<\delta+1<\alpha$.
According to Lemma \ref{lem:MPAFLdifferrlam}, there exists a constant
$\tilde{C}_{\lambda}>0$ such that for all $g\in\mathcal{B}_{b}(\mathbb{R}_{+}\times\mathbb{R}^{d})$,
\begin{equation}
|\nabla R_{\lambda}g(t,x+z)-\nabla R_{\lambda}g(t,x)|\leq\tilde{C}_{\lambda}\|g\||z|^{\delta},\quad(t,x)\in\hs,\ z\in\Rd,\label{eq2:MPAFLprop1}
\end{equation}
and $\tilde{C}_{\lambda}$ goes to $0$ as $\lambda\uparrow\infty$.

Let $g\in\mathcal{B}_{b}(\mathbb{R}_{+}\times\mathbb{R}^{d})$. For
all $(t,x)\in\hs$ and $z\in\Rd$, we have
\begin{align}
 & |R_{\lambda}g(t,x+z)-R_{\lambda}g(t,x)-z\cdot\nabla R_{\lambda}g(t,x)|\nonumber \\
 & \quad=\left|\int_{0}^{1}\nabla R_{\lambda}g(t,x+rz)\cdot zdr-z\cdot\nabla R_{\lambda}g(t,x)\right|\nonumber \\
 & \quad=\left|\int_{0}^{1}[\nabla R_{\lambda}g(t,x+rz)-\nabla R_{\lambda}g(t,x)]\cdot zdr\right|\nonumber \\
 & \quad\le|z|\int_{0}^{1}|\nabla R_{\lambda}g(t,x+rz)-\nabla R_{\lambda}g(t,x)|dr\overset{(\ref{eq2:MPAFLprop1})}{\le}\tilde{C}_{\lambda}\|g\||z|^{\delta+1}.\label{eq1:MPAFLprop1}
\end{align}
 So we obtain

\begin{align}
 & \int_{\mathbb{R}^{d}\backslash\{0\}}\left|R_{\lambda}g(t,x+z)-R_{\lambda}g(t,x)-\mathbf{1}_{\{|z|\le1\}}z\cdot\nabla R_{\lambda}g(t,x)\right|M(t,x,dz)\nonumber \\
 & \quad\le\int_{\{0<|z|\leq1\}}|R_{\lambda}g(t,x+z)-R_{\lambda}g(t,x)-z\cdot\nabla R_{\lambda}g(t,x)|M(t,x,dz)\nonumber \\
 & \qquad+{\displaystyle \int_{\{|z|>1\}}|R_{\lambda}g(t,x+z)-R_{\lambda}g(t,x)|M(t,x,dz)}\nonumber \\
 & \quad{\displaystyle \overset{(\ref{eq1:MPAFLprop1})}{\le}}\tilde{C}_{\lambda}\|g\|\int_{\{0<|z|\leq1\}}|z|^{\delta+1}M(t,x,dz)+2\Vert R_{\lambda}g\Vert\int_{\{|z|>1\}}1M(t,x,dz)\nonumber \\
{\displaystyle } & \quad\le(\tilde{C}_{\lambda}+2\lambda^{-1})\|g\|\sup_{t\geq0,x\in\mathbb{R}^{d}}\int_{\mathbb{R}^{d}\backslash\{0\}}1\wedge|z|^{\beta}M(t,x,dz).\label{esti: a>1, M R_lambda}
\end{align}
Hence $||KR_{\lambda}g||\leq k_{\lambda}||g||$ for all $\mathcal{B}_{b}(\mathbb{R}_{+}\times\mathbb{R}^{d})$.
\end{proof}
\begin{cor}
\label{cor MAPAFL:There-exists-}There exists $\lambda_{0}>0$ such
that for all $\lambda\ge\lambda_{0}$, we have $k_{\lambda}<1/2$
and
\begin{equation}
\left\Vert \sum_{i=0}^{\infty}R_{\lambda}(KR_{\lambda})^{i}g\right\Vert \le\sum_{i=0}^{\infty}\lambda^{-1}(k_{\lambda})^{i}\|g\|\le2\lambda^{-1}\|g\|,\quad g\in\mathcal{B}_{b}(\mathbb{R}_{+}\times\mathbb{R}^{d}).\label{esti2: kr_lambda}
\end{equation}
\end{cor}
According to Corollary \ref{cor MAPAFL:There-exists-}, for each $\lambda\ge\lambda_{0}$,
we can define
\begin{equation}
G_{\lambda}g:=\sum_{i=0}^{\infty}R_{\lambda}(KR_{\lambda})^{i}g,\quad g\in\mathcal{B}_{b}(\mathbb{R}_{+}\times\mathbb{R}^{d}).\label{defi: G_lambda}
\end{equation}

\begin{rem}
\label{rem: contin. of G_lambda}By (\ref{esti: kr_lambda}), (\ref{esti2: kr_lambda})
and (\ref{new new eq 0.5}), we see that if $\lambda\ge\lambda_{0}$
and $g\in\mathcal{B}_{b}(\mathbb{R}_{+}\times\mathbb{R}^{d})$, then
the function $\Rd\ni x\mapsto G_{\lambda}g(t,x)$ is bounded continuous
for each $t\ge0$.
\end{rem}
We have the following estimate of Krylov's type.
\begin{prop}
\label{prop: Krylov}Let $T>0$ and $g\in\mathcal{B}_{b}(\mathbb{R}_{+}\times\mathbb{R}^{d})$
be such that $supp(g)\subset[0,T]\times\Rd$ and $g\in L^{q}([0,T];L^{p}(\Rd))$
with $p,q>0$ and $d/p+\alpha/q<\alpha-\beta$, where $\beta\in(0,\alpha)$
is the constant in Assumption \ref{eq:MAPAFLassonM}. Then for each
$\lambda\ge\lambda_{0}$, there exists a constant $l_{\lambda}>0$,
independent of $g$ and $T$, such that
\begin{equation}
\|G_{\lambda}g\|\le l_{\lambda}\|g\|_{L^{q}([0,T];L^{p}(\Rd))}.\label{eq: Prop 4.3, to prove}
\end{equation}
Moreover, the constant $l_{\lambda}$ goes to $0$ as $\lambda\to\infty$.
\end{prop}
\begin{proof}
By (\ref{esti: lp of p_t}) and the same proof of \cite[Proposition 3.9 (i)]{jin2015weak},
we can find a constant $c_{\lambda}>0$, independent of $g$ and $T$,
such that
\begin{equation}
\|R_{\lambda}g\|\le c_{\lambda}\|g\|_{L^{q}([0,T];L^{p}(\Rd))},\label{esti: R_lambda g}
\end{equation}
where $c_{\lambda}$ goes to $0$ as $\lambda\to\infty$.

For $0<\alpha\leq1$, by (\ref{eq 1: Prop. 4.1}), (\ref{esti: R_lambda g})
and Lemma \ref{lem4:MPAFL} (ii), we have

\begin{align}
 & \int_{\mathbb{R}^{d}\backslash\{0\}}\left|R_{\lambda}g(t,x+z)-R_{\lambda}g(t,x)\right|M(t,x,dz)\nonumber \\
{\displaystyle } & \quad\le N_{\lambda}\|g\|_{L^{q}([0,T];L^{p}(\Rd))}\int_{\{0<|z|\leq1\}}|z|^{\beta}M(t,x,dz)\nonumber \\
 & \qquad+2c_{\lambda}\|g\|_{L^{q}([0,T];L^{p}(\Rd))}\int_{\{|z|>1\}}1M(t,x,dz)\nonumber \\
{\displaystyle } & \quad\le(N_{\lambda}+2c_{\lambda})\|g\|_{L^{q}([0,T];L^{p}(\Rd))}\sup_{t\geq0,x\in\mathbb{R}^{d}}\int_{\mathbb{R}^{d}\backslash\{0\}}1\wedge|z|^{\beta}M(t,x,dz).\label{esti: lp, M R_lambda}
\end{align}
For $1<\alpha<2$, similarly to (\ref{esti: a>1, M R_lambda}), we
obtain

\begin{align}
 & \int_{\mathbb{R}^{d}\backslash\{0\}}\left|R_{\lambda}g(t,x+z)-R_{\lambda}g(t,x)-\mathbf{1}_{\{|z|\le1\}}z\cdot\nabla R_{\lambda}g(t,x)\right|M(t,x,dz)\nonumber \\
 & \quad\le(\tilde{N}_{\lambda}+2c_{\lambda})\|g\|_{L^{q}([0,T];L^{p}(\Rd))}\sup_{t\geq0,x\in\mathbb{R}^{d}}\int_{\mathbb{R}^{d}\backslash\{0\}}1\wedge|z|^{\beta}M(t,x,dz),\label{esti: a>1,  lp, M R_lambda}
\end{align}
where $\tilde{N}_{\lambda}>0$ is the constant from Lemma \ref{lem:MPAFLdifferrlam}
(ii). Summarizing (\ref{esti: lp, M R_lambda}) and (\ref{esti: a>1,  lp, M R_lambda}),
we obtain that for all $\alpha\in(0,2)$,
\begin{equation}
\|KR_{\lambda}g\|\leq\tilde{c}_{\lambda}\|g\|_{L^{q}([0,T];L^{p}(\Rd))},\label{esti: Prop. 4.3, KR_lambda}
\end{equation}
where
\begin{equation}
\tilde{c}_{\lambda}:=\begin{cases}
(N_{\lambda}+2c_{\lambda})\sup_{t\geq0,x\in\mathbb{R}^{d}}\int_{\mathbb{R}^{d}\backslash\{0\}}1\wedge|z|^{\beta}M(t,x,dz), & 0<\alpha\le1,\\
(\tilde{N}_{\lambda}+2c_{\lambda})\sup_{t\geq0,x\in\mathbb{R}^{d}}\int_{\mathbb{R}^{d}\backslash\{0\}}1\wedge|z|^{\beta}M(t,x,dz), & 1<\alpha<2.
\end{cases}\label{defi: c tilde _lambda}
\end{equation}
By (\ref{esti: R_lambda g}), (\ref{esti: Prop. 4.3, KR_lambda})
and Lemma \ref{prop MPAFL:For-any-}, we obtain that for all $i\in\mathbb{N}$,
\[
\|R_{\lambda}(KR_{\lambda})^{i}g\|\leq c_{\lambda}\left(k_{\lambda}\right)^{i-1}\|KR_{\lambda}g\|\le c_{\lambda}\left(k_{\lambda}\right)^{i-1}\tilde{c}_{\lambda}\|g\|_{L^{q}([0,T];L^{p}(\Rd))},
\]
which implies that for $\lambda\ge\lambda_{0}$,
\begin{align*}
\|G_{\lambda}g\|\le\sum_{i=0}^{\infty}\|R_{\lambda}(KR_{\lambda})^{i}g\| & \leq c_{\lambda}\left(1+\sum_{i=1}^{\infty}\tilde{c}_{\lambda}\left(k_{\lambda}\right)^{i-1}\right)\|g\|_{L^{q}([0,T];L^{p}(\Rd))}\\
 & \le c_{\lambda}\left(1+2\tilde{c}_{\lambda}\right)\|g\|_{L^{q}([0,T];L^{p}(\Rd))}.
\end{align*}
So (\ref{eq: Prop 4.3, to prove}) holds with
\begin{equation}
l_{\lambda}:=c_{\lambda}\left(1+2\tilde{c}_{\lambda}\right)>0.\label{defi: l_lambda}
\end{equation}
 Since $c_{\lambda}$, $N_{\lambda}$ and $\tilde{N}_{\lambda}$ all
converge to $0$ as $\lambda\to\infty$, we see that $\lim_{\lambda\to\infty}l_{\lambda}=0$.
\end{proof}

\section{Well-posedness of the martingale problem for $\mathcal{L}_{t}$ }

In this section we prove our main result, namely, the martingale problem
for $\mathcal{L}_{t}$ is well-posed. In view of (\ref{defi: G_lambda}),
the uniqueness problem can be solved by standard perturbation arguments.
To obtain existence, we will first consider smooth approximations
of $\mathcal{L}_{t}$ and then construct a solution to the martingale
problem for $\mathcal{L}_{t}$ by weak convergence of probability
measures.

Let $\phi\in C_{0}^{\infty}(\Rd)$ be such that $0\le\phi\le1$, $\int_{\Rd}\phi(x)dx=1$
and $\phi(x)=0$ for $|x|\ge1$. Define $\phi_{n}(x):=n^{d}\phi(nx)$,
$x\in\Rd$. Given $n\in\mathbb{N}$, define $M_{n}(t,x,\cdot)$ as
the kernel obtained by mollifying $M(t,x,\cdot)$ through $\phi_{n}$,
that is,

\[
M_{n}(t,x,B):=\int_{\Rd}M(t,x-z,B)\phi_{n}(z)dz,\quad B\in\mathcal{{B}}(\Rd).
\]
So $M_{n}(t,x,\cdot)$ is a kernel from $\hs$ to $\mathcal{B}\left(\mathbb{R}^{d}\backslash\{0\}\right)$
and $M_{n}(t,x,\cdot)$ is a Lévy measure on $\mathbb{R}^{d}\backslash\{0\}$
for each $(t,x)\in\hs$. By Fubini's theorem, we have that for all
$(t,x)\in\hs$ and $n\in\mathbb{N},$
\begin{align}
\int_{\mathbb{R}^{d}\backslash\{0\}}1\wedge|y|^{\beta}M_{n}(t,x,dy) & =\int_{\Rd}\left(\int_{\mathbb{R}^{d}\backslash\{0\}}1\wedge|y|^{\beta}M(t,x-z,dy)\right)\phi_{n}(z)dz\nonumber \\
 & \le\sup_{t\geq0,x\in\mathbb{R}^{d}}\int_{\mathbb{R}^{d}\backslash\{0\}}1\wedge|y|^{\beta}M(t,x,dy)<\infty.\label{condition: sup M_n}
\end{align}
Define
\[
K_{n,t}f(x):=\int_{\mathbb{R}^{d}\backslash\{0\}}[f(x+y)-f(x)-\mathbf{1}_{\alpha>1}\mathbf{1}_{\{|y|\le1\}}y\cdot\nabla f(x)]M_{n}(t,x,dy).
\]

\begin{lem}
\label{lem: delta }Let $f\in C_{b}^{3}(\Rd)$ be arbitrary. Then
for all $(t,x)\in\hs$, we have
\[
\left|K_{n,t}f(x)-K_{t}f\ast\phi_{n}(x)\right|\le4n^{-1}\left\Vert f\right\Vert _{C_{b}^{3}(\Rd)}\sup_{t,x}\int_{\mathbb{R}^{d}\backslash\{0\}}(1\wedge|h|^{\beta})M(t,x,dh).
\]
\end{lem}
\begin{proof}
First note that
\begin{equation}
\int_{\Rd}|y|\phi_{n}(y)dy=\int_{\{|y|\le1/n\}}|y|\phi_{n}(y)dy\le n^{-1}.\label{eq, Lemma 5.1: *}
\end{equation}
Let
\begin{align*}
\Delta_{n,t}f(x):= & K_{n,t}f(x)-K_{t}f\ast\phi_{n}(x).
\end{align*}
(i) For the case $0<\alpha\le1$, we have
\begin{align}
\Delta_{n,t}f(x) & =\int_{\mathbb{R}^{d}\backslash\{0\}}\int_{\Rd}[f(x+h)-f(x)]M(t,x-y,dh)\phi_{n}(y)dy\nonumber \\
 & \quad-\int_{\mathbb{R}^{d}\backslash\{0\}}\int_{\Rd}[f(x-y+h)-f(x-y)]M(t,x-y,dh)\phi_{n}(y)dy.\label{eq, Lemma 5.1: **}
\end{align}
Since
\begin{align*}
 & \left|f(x+h)-f(x-y+h)-f(x)+f(x-y)\right|\\
 & \quad=\left|\int_{0}^{1}\left[\nabla f(x+h-y+ry)-\nabla f(x-y+ry)\right]\cdot ydr\right|\\
 & \quad\le2|y|\left(1\wedge|h|\right)\left\Vert f\right\Vert _{C_{b}^{2}(\Rd)},
\end{align*}
it follows from (\ref{eq, Lemma 5.1: **}) that
\begin{align*}
|\Delta_{n,t}f(x)| & \le2\left\Vert f\right\Vert _{C_{b}^{2}(\Rd)}\int_{\mathbb{R}^{d}\backslash\{0\}}\int_{\Rd}(1\wedge|h|)|y|M(t,x-y,dh)\phi_{n}(y)dy\\
 & \le2\left\Vert f\right\Vert _{C_{b}^{2}(\Rd)}\sup_{t,x}\int_{\mathbb{R}^{d}\backslash\{0\}}(1\wedge|h|^{\beta})M(t,x,dh)\int_{\Rd}|y|\phi_{n}(y)dy\\
 & {\displaystyle \overset{(\ref{eq, Lemma 5.1: *})}{\le}}2n^{-1}\left\Vert f\right\Vert _{C_{b}^{2}(\Rd)}\sup_{t,x}\int_{\mathbb{R}^{d}\backslash\{0\}}(1\wedge|h|^{\beta})M(t,x,dh).
\end{align*}
(ii) For $1<\alpha<2$, we have
\begin{align*}
\Delta_{n,t}f(x) & =\int_{\mathbb{R}^{d}\backslash\{0\}}\int_{\Rd}[f(x+h)-f(x)-\mathbf{1}_{\{|h|\le1\}}h\cdot\nabla f(x)]M(t,x-y,dh)\phi_{n}(y)dy\\
 & \quad-\int_{\mathbb{R}^{d}\backslash\{0\}}\int_{\Rd}[f(x-y+h)-f(x-y)-\mathbf{1}_{\{|h|\le1\}}h\cdot\nabla f(x-y)]\\
 & \qquad\times M(t,x-y,dh)\phi_{n}(y)dy.
\end{align*}
If $|h|>1$, then $\left|f(x+h)-f(x-y+h)-f(x)+f(x-y)\right|\le4\left\Vert f\right\Vert $;
for $0<|h|\le1$, we have
\begin{align*}
 & \left|f(x+h)-f(x)-h\cdot\nabla f(x)-f(x-y+h)+f(x-y)+h\cdot\nabla f(x-y)\right|\\
 & \quad=\left|\int_{0}^{1}\left[\nabla f(x+rh)-\nabla f(x)-\nabla f(x-y+rh)+\nabla f(x-y)\right]\cdot hdr\right|\\
 & \quad=\left|\int_{0}^{1}\left[\int_{0}^{1}\left(\nabla^{2}f(x-y+rh+r'y)-\nabla^{2}f(x-y+r'y)\right)\cdot ydr'\right]\cdot hdr\right|\\
 & \quad\le|y||h|^{2}\left\Vert f\right\Vert _{C_{b}^{3}(\Rd)}.
\end{align*}
So
\begin{align*}
|\Delta_{n,t}f(x)| & \le4\left\Vert f\right\Vert _{C_{b}^{3}(\Rd)}\int_{\mathbb{R}^{d}\backslash\{0\}}\int_{\Rd}(1\wedge|h|^{2})|y|M(t,x-y,dh)\phi_{n}(y)dy\\
 & \le4\left\Vert f\right\Vert _{C_{b}^{3}(\Rd)}\sup_{t,x}\int_{\mathbb{R}^{d}\backslash\{0\}}(1\wedge|h|^{\beta})M(t,x,dh)\int_{\Rd}|y|\phi_{n}(y)dy\\
 & {\displaystyle \overset{(\ref{eq, Lemma 5.1: *})}{\le}}4n^{-1}\left\Vert f\right\Vert _{C_{b}^{3}(\Rd)}\sup_{t,x}\int_{\mathbb{R}^{d}\backslash\{0\}}(1\wedge|h|^{\beta})M(t,x,dh).
\end{align*}
\end{proof}
\begin{lem}
\label{lem: existence L_n}For each $(s,x)\in\hs$, there exists at
least one solution to the martingale problem for $\mathcal{L}_{n,t}=L+K_{n,t}$
starting from $(s,x)$.
\end{lem}
\begin{proof}
To prove the solvability of the martingale problem for $\mathcal{L}_{n,t}$,
we use the same argument as in \cite[Theorem~(2.2)]{MR0433614}. Let
$\varphi\in C_{b}^{\infty}(\Rd)$ be such that $0\le\varphi\le1$,
$\varphi(y)=0$ for $|y|\le1/2,$ and $\varphi(y)=1$ for $|y|\ge1$.
For $0<\delta<1$, let $\varphi_{\delta}(y):=\varphi(y/\delta)$ and
define the kernel $M_{n}^{\delta}(t,x,\cdot)$ by
\[
M_{n}^{\delta}(t,x,dy):=\varphi_{\delta}(y)M_{n}(t,x,dy).
\]
 Set $c_{\delta}(t,x)=\mathbf{1}_{\alpha>1}\int_{\{|y|\le1\}}yM_{n}^{\delta}(t,x,dy)$.
Since
\begin{align*}
c_{\delta}(t,x) & =\mathbf{1}_{\alpha>1}\int_{\{|y|\le1\}}y\varphi_{\delta}(y)M_{n}(t,x,dy)\\
 & =\mathbf{1}_{\alpha>1}\int_{\Rd}\left(\int_{\{|y|\le1\}}y\varphi_{\delta}(y)M(t,x-z,dy)\right)\phi_{n}(z)dz\\
 & =\mathbf{1}_{\alpha>1}\int_{\Rd}\left(\int_{\{|y|\le1\}}y\varphi_{\delta}(y)M(t,z,dy)\right)\phi_{n}(x-z)dz,
\end{align*}
we see that $|\nabla_{x}c_{\delta}(t,x)|$ is bounded on $\hs$. Hence
$c_{\delta}(t,x)$ is globally Lipschitz continuous in $x$. Define
the differential operator $A_{t}^{\delta}$ by
\[
A_{t}^{\delta}f(x):=\sum_{i,j=1}^{d}a_{ij}\frac{\partial^{2}}{\partial x_{i}\partial x_{j}}f(x)+b\cdot\nabla f(x)-c_{\delta}(t,x)\cdot\nabla f(x).
\]
By the Lipschitz continuity (in the space variable $x$) of the coefficients
of $A_{t}^{\delta}$, there is for each $(s,x)$ a unique solution
$\mathbf{{Q}}_{\delta}^{s,x}$ to the martingale problem for $A_{t}^{\delta}$
starting from $(s,x)$, see, e.g., \cite[Theorem 5.1.1 and Corollary 5.1.3]{MR2190038}.
By \cite[Theorem 5.1.4]{MR2190038}, the mapping $(s,x)\mapsto\mathbf{{Q}}_{\delta}^{s,x}(E)$
is measurable for all $E\in\mathcal{D}$. Note that $A_{t}^{\delta}f(x)+\int_{\mathbb{R}^{d}\backslash\{0\}}[f(x+y)-f(x)]M_{n}^{\delta}(t,x,dy)=Lf(x)+K_{n,t}^{\delta}f(x)$,
where
\[
K_{n,t}^{\delta}f(x):=\int_{\mathbb{R}^{d}\backslash\{0\}}[f(x+y)-f(x)-\mathbf{1}_{\alpha>1}\mathbf{1}_{\{|y|\le1\}}y\cdot\nabla f(x)]M_{n}^{\delta}(t,x,dy).
\]
It follows from \cite[Theorem~(2.1)]{MR0433614} that the martingale
problem for $L+K_{n,t}^{\delta}$ is solvable. For $f\in C_{0}^{\infty}(\Rd)$,
we have
\begin{align*}
\left|K_{n,t}^{\delta}f(x)-K_{n,t}f(x)\right| & \le\int_{\{|y|\le\delta\}}\left|f(x+y)-f(x)-\mathbf{1}_{\alpha>1}y\cdot\nabla f(x)\right|M_{n}(t,x,dy)\\
 & \le\left\Vert f\right\Vert _{C_{b}^{2}(\Rd)}\int_{\{|y|\le\delta\}}\left(\mathbf{1}_{\alpha\le1}|y|+\mathbf{1}_{\alpha>1}|y|^{2}\right)M_{n}(t,x,dy)\\
 & \le\left\Vert f\right\Vert _{C_{b}^{2}(\Rd)}\int_{\{|y|\le\delta\}}|y|^{\alpha}M_{n}(t,x,dy)\\
 & \le\delta^{\alpha-\beta}\left\Vert f\right\Vert _{C_{b}^{2}(\Rd)}\int_{\{|y|\le\delta\}}|y|^{\beta}M_{n}(t,x,dy)\\
 & {\displaystyle \overset{(\ref{condition: sup M_n})}{\le}}\delta^{\alpha-\beta}\left\Vert f\right\Vert _{C_{b}^{2}(\Rd)}\sup_{t\geq0,x\in\mathbb{R}^{d}}\int_{\mathbb{R}^{d}\backslash\{0\}}1\wedge|y|^{\beta}M(t,x,dy),
\end{align*}
which implies that $K_{n,t}^{\delta}f\to K_{n,t}f$ uniformly as $\delta\to0$.
The rest of the proof goes in the same way as in \cite[Theorem~(2.2)]{MR0433614}.
We omit the details.
\end{proof}
Recall that $\lambda_{0}>0$ is the constant given in Corollary \ref{cor MAPAFL:There-exists-}.
\begin{lem}
\label{lem: resolvent L_n}Let $(s,x)\in\hs$ and $\mathbf{{P}}_{n}^{s,x}$
be a solution to the martingale problem for $\mathcal{L}_{n,t}=L+K_{n,t}$
starting from $(s,x)$. Then for any $\lambda\ge\lambda_{0}$ and
$g\in\mathcal{B}_{b}(\hs)$, we have
\begin{equation}
\mathbf{E}_{n}^{s,x}\Big[\int_{s}^{\infty}e^{-\lambda(t-s)}g(t,X_{t})dt\Big]=\sum_{k=0}^{\infty}R_{\lambda}(K_{n}R_{\lambda})^{k}g(s,x),\label{eqressn}
\end{equation}
where $\mathbf{E}_{n}^{s,x}[\cdot]$ denotes the expectation with
respect to the measure $\mathbf{{P}}_{n}^{s,x}$ and $K_{n}R_{\lambda}$
is defined by
\begin{align}
K_{n}R_{\lambda}g(t,x) & :=\int_{\mathbb{R}^{d}\backslash\{0\}}[R_{\lambda}g(t,x+y)-R_{\lambda}g(t,x)\nonumber \\
 & \qquad-\mathbf{1}_{\alpha>1}\mathbf{1}_{\{|y|\le1\}}y\cdot\nabla R_{\lambda}g(t,x)]M_{n}(t,x,dy),\quad(t,x)\in\hs.\label{defi: KR_lambda-1}
\end{align}
\end{lem}
\begin{proof}
For $\lambda>0$ and $f\in\mathcal{B}_{b}(\hs)$, define
\[
V_{n}^{\lambda}f:=\mathbf{E}_{n}^{s,x}\Big[\int_{s}^{\infty}e^{-\lambda(t-s)}f(t,X_{t})dt\Big].
\]
For $f\in C_{b}^{1,2}(\hs)$, we know that
\begin{align*}
 & f(t,X_{t})-f(s,X_{s})\\
= & ``Martingale"+\int_{s}^{t}(\frac{\partial f}{\partial u}+\mathcal{L}_{n,u}f)(u,X_{u})du.
\end{align*}
Taking expectations of both sides of the above equality gives
\begin{equation}
\mathbf{E}_{n}^{s,x}[f(t,X_{t})]-f(s,x)=\mathbf{E}_{n}^{s,x}\Big[\int_{s}^{t}(\frac{\partial f}{\partial u}+\mathcal{L}_{n,u}f)(u,X_{u})du\Big].\label{thmunieq0}
\end{equation}
Multiplying both sides of (\ref{thmunieq0}) by $e^{-\lambda(t-s)}$,
integrating with respect to $t$ from $0$ to $\infty$ and then applying
Fubini's theorem, we get
\begin{align}
 & \mathbf{E}_{n}^{s,x}\Big[\int_{s}^{\infty}e^{-\lambda(t-s)}f(t,X_{t})dt\Big]\nonumber \\
= & \frac{1}{\lambda}f(s,x)+\mathbf{E}_{n}^{s,x}\Big[\int_{s}^{\infty}e^{-\lambda(t-s)}\int_{s}^{t}\big(\frac{\partial f}{\partial u}+\mathcal{L}_{n,u}f\big)(u,X_{u})dudt\Big]\nonumber \\
= & \frac{1}{\lambda}f(s,x)+\frac{1}{\lambda}\mathbf{E}_{n}^{s,x}\Big[\int_{s}^{\infty}e^{-\lambda(u-s)}\big(\frac{\partial f}{\partial u}+\mathcal{L}_{n,u}f\big)(u,X_{u})du\Big].\label{eq: Fubini}
\end{align}
Therefore, for $f\in C_{b}^{1,2}(\hs)$,
\begin{equation}
\lambda V_{n}^{\lambda}f=f(s,x)+V_{n}^{\lambda}\Big(\frac{\partial f}{\partial t}+\mathcal{L}_{n,t}f\Big).\label{thmunieq101}
\end{equation}
If $g\in C_{b}^{1,2}(\hs)$, then $f:=R_{\lambda}g\in C_{b}^{1,2}(\hs)$
and
\begin{equation}
{\displaystyle \lambda f(t,y)-Lf(t,y)-\frac{\partial}{\partial t}f(t,y)=g(t,y),\quad(t,y)\in\hs,}\label{eq for para. resol. eq.: MPAFL}
\end{equation}
see, e.g., the proof of \cite[Proposition~3.8]{jin2015weak}. Substituting
this $f$ in (\ref{thmunieq101}), we obtain $V_{n}^{\lambda}g=R_{\lambda}g(s,x)+V_{n}^{\lambda}(K_{n}R_{\lambda}g)$
for $g\in C_{b}^{1,2}(\hs)$. If $g\in C_{0}(\mathbb{R}_{+}\times\mathbb{R}^{d})$,
namely, $g$ is continuous with compact support, then there exist
$g_{k}\in C_{b}^{1,2}(\mathbb{R}_{+}\times\mathbb{R}^{d})$ such that
$g_{k}\rightarrow g$ boundedly and uniformly as $k\rightarrow\infty$.
\textcolor{black}{It follows from (}\ref{esti: kr_lambda}) that\textcolor{black}{{}
$K_{n}R_{\lambda}g_{k}\rightarrow KR_{\lambda}g$ boundedly and pointwise
as $k\rightarrow\infty$.} By the dominated convergence theorem, we
have
\begin{align*}
V_{n}^{\lambda}g=\lim_{k\to\infty}V_{n}^{\lambda}g_{k} & =\lim_{k\to\infty}\left\{ R_{\lambda}g_{k}(s,x)+V_{n}^{\lambda}(K_{n}R_{\lambda}g_{k})\right\} \\
 & =R_{\lambda}g(s,x)+V_{n}^{\lambda}(K_{n}R_{\lambda}g),\quad g\in C_{0}(\mathbb{R}_{+}\times\mathbb{R}^{d}).
\end{align*}
Then by a standard monotone class argument, we arrive at
\begin{equation}
V_{n}^{\lambda}g=R_{\lambda}g(s,x)+V_{n}^{\lambda}(K_{n}R_{\lambda}g),\quad g\in\mathcal{B}_{b}(\hs).\label{WUSeqlemma43vnl}
\end{equation}
For $g\in\mathcal{B}_{b}(\hs)$, we thus have
\begin{align}
V_{n}^{\lambda}g= & R_{\lambda}g(s,x)+V_{n}^{\lambda}(K_{n}R_{\lambda}g)\nonumber \\
{\displaystyle \overset{(\ref{WUSeqlemma43vnl})}{=}} & R_{\lambda}g(s,x)+R_{\lambda}K_{n}R_{\lambda}g(s,x)+V_{n}^{\lambda}(K_{n}R_{\lambda})^{2}g\nonumber \\
= & \cdots=\sum_{k=0}^{i}R_{\lambda}(K_{n}R_{\lambda})^{k}g(s,x)+V_{n}^{\lambda}(K_{n}R_{\lambda})^{i+1}g.\label{eqsnforn}
\end{align}

Let $k_{\lambda}>0$ be as in (\ref{defi: k_lambda}). By (\ref{condition: sup M_n})
and Proposition \ref{prop MPAFL:For-any-}, we have
\[
\|K_{n}R_{\lambda}g\|\le k_{\lambda}\Vert g\Vert,\quad\forall n\in\BN,\ g\in\mathcal{B}_{b}(\hs).
\]
According to Corollary \ref{cor MAPAFL:There-exists-}, we have $k_{\lambda}<1/2$
for $\lambda\ge\lambda_{0}$. Therefore, for all $i,n\in\mathbb{N}$,
$\lambda\ge\lambda_{0}$ and $g\in\mathcal{B}_{b}(\hs)$,
\[
|V_{n}^{\lambda}(K_{n}R_{\lambda})^{i}g|\le\lambda^{-1}\big(k_{\lambda}\big)^{i}\Vert g\Vert\le\lambda^{-1}2^{-i}\Vert g\Vert
\]
and
\[
\left\Vert R_{\lambda}(K_{n}R_{\lambda})^{i}g\right\Vert \le\lambda^{-1}\big(k_{\lambda}\big)^{i}\Vert g\Vert\le\lambda^{-1}2^{-i}\Vert g\Vert.
\]
Letting $i\to\infty$ in (\ref{eqsnforn}) gives (\ref{eqressn}).
This completes the proof.
\end{proof}
\begin{rem}
In view of (\ref{condition: sup M_n}), we can repeat the proof of
Proposition \ref{prop: Krylov} to obtain that for each $\lambda\ge\lambda_{0}$,
\begin{equation}
\left\Vert \sum_{k=0}^{\infty}R_{\lambda}(K_{n}R_{\lambda})^{k}g\right\Vert \le l_{\lambda}\|g\|_{L^{q}([0,T];L^{p}(\Rd))},\label{defi 2: l_lambda}
\end{equation}
where $d/p+\alpha/q<\alpha-\beta$ and $g\in\mathcal{B}_{b}([0,\infty)\times\mathbb{R}^{d})\cap L^{q}([0,T];L^{p}(\Rd))$
is an arbitrary function satisfying $supp(g)\subset[0,T]\times\Rd$.
Indeed, by (\ref{defi: c tilde _lambda}) and (\ref{defi: l_lambda}),
the constant $l_{\lambda}>0$ here can be chosen to be the same as
in (\ref{eq: Prop 4.3, to prove}). In particular, $l_{\lambda}$
in (\ref{defi 2: l_lambda}) is independent of $n\in\BN.$
\end{rem}
\begin{cor}
\label{cor: krylov esti for p_n}Let $\mathbf{{P}}_{n}^{s,x}$ be
as in Lemma \ref{lem: resolvent L_n}. Let $p>(d+\alpha)/(\alpha-\beta)$.
For each $T>s$, there exists a constant $C_{T}>0$, which is independent
of $n$, such that
\[
\mathbf{E}_{n}^{s,x}\Big[\int_{s}^{T}|f(t,X_{t})|dt\Big]\le C_{T}\|f\|_{L^{p}\left([0,T]\times\Rd\right)},\quad\forall f\in L^{p}\left([0,T]\times\Rd\right).
\]
\end{cor}
\begin{proof}
Let $f\in\mathcal{B}_{b}([0,T]\times\mathbb{R}^{d})\cap L^{p}\left([0,T]\times\Rd\right)$.
Applying (\ref{defi 2: l_lambda}) with $p=q>(d+\alpha)/(\alpha-\beta)$,
we get
\begin{align*}
\mathbf{E}_{n}^{s,x}\Big[\int_{s}^{T}|f(t,X_{t})|dt\Big] & \le e^{\lambda_{0}(T-s)}\mathbf{E}_{n}^{s,x}\Big[\int_{s}^{\infty}e^{-\lambda_{0}(t-s)}\mathbf{1}_{[0,T]}(t)|f(t,X_{t})|dt\Big]\\
 & \overset{(\ref{eqressn})}{\le}e^{\lambda_{0}(T-s)}\sum_{k=0}^{\infty}R_{\lambda_{0}}(K_{n}R_{\lambda_{0}})^{k}\left(\mathbf{1}_{[0,T]}(t)|f|\right)(s,x)\\
 & \overset{(\ref{defi 2: l_lambda})}{\le}l_{\lambda_{0}}e^{\lambda_{0}(T-s)}\|f\|_{L^{p}([0,T]\times\Rd)}.
\end{align*}
 For a general $f\in L^{p}\left([0,T]\times\Rd\right)$, the assertion
follows from the monotone convergence theorem.
\end{proof}
We are now ready to prove Theorem \ref{thm: main}.

\subsection*{Proof of Theorem \ref{thm: main}}

``\emph{Existence}'': \textcolor{black}{Let $(s,x)\in\hs$ be fixed.
It follows from Lemma }\ref{lem: existence L_n}\textcolor{black}{{}
that there exists a solution $\mathbf{{P}}_{n}^{s,x}$ to the martingale
problem for $\mathcal{L}_{n,t}=L+K_{n,t}$ starting from $(s,x)$. }

By (\ref{condition: sup M_n}) and \cite[Theorem~(A.1)]{MR0433614},
the family $\{\mathbf{{P}}_{n}^{s,x},\ n\in\mathbb{N}\}$ is tight.
Let $\mathbf{{P}}^{s,x}$ be a limit point of $\{\mathbf{{P}}_{n}^{s,x},\ n\in\mathbb{N}\}$.
Then there exists a subsequence of $\left(\mathbf{{P}}_{n}^{s,x}\right)_{n\in\mathbb{N}}$
which converges weekly to $\mathbf{{P}}^{s,x}$. For simplicity, we
denote this subsequence still by $\left(\mathbf{{P}}_{n}^{s,x}\right)_{n\in\mathbb{N}}$.

We next show that $\mathbf{{P}}^{s,x}$ is a solution to the martingale
problem for $\mathcal{L}_{t}$ starting from $(s,x)$. Let $f\in C_{0}^{\infty}(\mathbb{R}^{d})$
be arbitrary. By \cite[Theorem~(1.1)]{MR0433614}, it suffices to
show that
\[
f(X_{t})-\int_{s}^{t}\mathcal{L}_{u}f(X_{u})du
\]
is a $\mathbf{P}^{s,x}$-martingale after time $s$. Suppose $s\le t_{1}\le t_{2}$,
$0\le r_{1}\le\cdots\le r_{l}\le t_{1}$ and $g_{1},\cdots,g_{l}\in C_{0}(\mathbb{R}^{d})$,
where $l\in\BN$. Set $Y=\prod_{j=1}^{l}g_{j}(X_{r_{j}}).$ It reduces
to show that
\begin{equation}
\mathbf{E}^{s,x}\Big[Y\big(f(X_{t_{2}})-f(X_{t_{1}})-\int_{t_{1}}^{t_{2}}\mathcal{L}_{u}f(X_{u})du\big)\Big]=0.\label{eq, thm 5.3: to show}
\end{equation}

We will complete the proof of (\ref{eq, thm 5.3: to show}) in four
steps. Firstly, note that by \cite[~Chap.~3,~Lemma~7.7]{MR838085},
there exists a countable set $I\subset\BR_{+}$ such that
\begin{equation}
\mathbf{P}^{s,x}(X_{t-}=X_{t})=1,\qquad\forall t\in\BR_{+}\setminus I.\label{WUSeqthm4545}
\end{equation}
Since $\BR_{+}\setminus I$ is dense in $\BR_{+}$ and $t\mapsto X_{t}(\omega)$,
$\omega\in D$, is right-continuous, it is enough to show (\ref{eq, thm 5.3: to show})
by additionally assuming that
\begin{equation}
r_{1},\cdots,r_{l},t_{1},t_{2}\in\BR_{+}\setminus I.\label{extra ass, I}
\end{equation}
So from now on, we assume that (\ref{extra ass, I}) is true.

``\textit{Step 1}'': We establish an estimate of Krylov's type for
$\mathbf{P}^{s,x}$. Let $p>(d+\alpha)/(\alpha-\beta)$.\textcolor{black}{{}
By Corollary }\ref{cor: krylov esti for p_n}, for each $T>s$, there
exists a constant $C_{T}>0$ such that
\begin{equation}
\sup_{n\in\BN}\mathbf{E}_{n}^{s,x}\Big[\int_{s}^{T}|f(t,X_{t})|dt\Big]\le C_{T}\|f\|_{L^{p}\left([0,T]\times\Rd\right)},\quad\forall f\in L^{p}\left([0,T]\times\Rd\right).\label{eq: kry p_n}
\end{equation}
It follows that for each $T>s$,
\begin{equation}
\mathbf{E}^{s,x}\Big[\int_{s}^{T}|f(t,X_{t})|dt\Big]\le C_{T}\|f\|_{L^{p}\left([0,T]\times\Rd\right)},\quad\forall f\in L^{p}\left([0,T]\times\Rd\right).\label{eq: kry P}
\end{equation}
Indeed, if $f\in C_{0}([0,T]\times\Rd)$, namely, $f$ is continuous
on $[0,T]\times\Rd$ with compact support, then (\ref{eq: kry P})
follows from (\ref{eq: kry p_n}) and the weak convergence of $\mathbf{{P}}_{n}^{s,x}$
to $\mathbf{{P}}^{s,x}$. By a standard monotone class argument, we
obtain (\ref{eq: kry P}) for all $f\in L^{p}\left([0,T]\times\Rd\right)$.

``\textit{Step 2}'': We show that
\begin{align}
 & \lim_{n\to\infty}\mathbf{E}_{n}^{s,x}\Big[Y\big(f(X_{t_{2}})-f(X_{t_{1}})-\int_{t_{1}}^{t_{2}}Lf(X_{u})du\big)\Big]\nonumber \\
 & \quad=\mathbf{E}^{s,x}\Big[Y\big(f(X_{t_{2}})-f(X_{t_{1}})-\int_{t_{1}}^{t_{2}}Lf(X_{u})du\big)\Big].\label{eq 3, thm 5.3: to show}
\end{align}
By Skorokhod's representation theorem, there exists a probability
space $(\Omega,\mathcal{A},\mathbf{Q})$ and random elements $\xi,\xi_{1},\cdots,\xi_{n},\cdots:\Omega\to D$
such that $\mathbf{{P}}_{n}^{s,x}=\mathbf{Q}\circ\xi_{n}^{-1}$, $\mathbf{{P}}^{s,x}=\mathbf{Q}\circ\xi^{-1}$
and $d(\xi_{n},\xi)\to0$ $\mathbf{Q}-\mathrm{a.s.}$, where $d$
is the Skorokhod metric on $D$. It follows from (\ref{WUSeqthm4545})
and \cite[~Chap.~3,~Prop.~5.2]{MR838085} that
\begin{equation}
\lim_{n\to\infty}X_{t}(\xi_{n})=X_{t}(\xi)\quad\mathbf{Q}\mbox{-a.s.},\qquad\forall t\in\BR_{+}\setminus I.\label{WUSeqthm4546}
\end{equation}
Let $\mathbf{E}[\cdot]$ denote the expectation with respect to the
measure $\mathbf{Q}$. By (\ref{extra ass, I}) and the dominated
convergence theorem, we have
\begin{align*}
 & \lim_{n\to\infty}\mathbf{E}\Big[Y(\xi_{n})\Big\{ f(X_{t_{2}}(\xi_{n}))-f(X_{t_{1}}(\xi_{n}))-\int_{t_{1}}^{t_{2}}Lf(X_{u}(\xi_{n}))du\Big\}\Big]\\
 & {\displaystyle \quad\overset{(\ref{WUSeqthm4546})}{=}}\mathbf{E}\Big[Y(\xi)\Big\{ f(X_{t_{2}}(\xi))-f(X_{t_{1}}(\xi))-\int_{t_{1}}^{t_{2}}Lf(X_{u}(\xi))du\Big\}\Big],
\end{align*}
which implies (\ref{eq 3, thm 5.3: to show}).

``\textit{Step 3}'': We show that
\begin{equation}
\lim_{n\to\infty}\mathbf{E}_{n}^{s,x}\Big[Y\int_{t_{1}}^{t_{2}}K_{n,u}f(X_{u})du\Big]=\mathbf{E}^{s,x}\Big[Y\int_{t_{1}}^{t_{2}}K_{u}f(X_{u})du\Big].\label{eq 2, thm 5.3: to show}
\end{equation}
Note that $Y$ is bounded. Let $C_{Y}:=\sup_{\omega\in D}|Y(\omega)|<\infty.$
For $r>0$ let $\chi_{r}$ be a continuous non-negative function on
$\Rd$ with $\chi_{r}(x)=1$ for $|x|\le r$ , $\chi_{r}(x)=0$ for
$|x|>r+1$ and $0\le$$\chi_{r}(x)\le1$ for $r<|x|\le r+1$; moreover,
we can choose $\chi_{r}$ such that $\chi_{r}$ is monotone in $r$,
namely, $\chi_{r_{1}}\le\chi_{r_{2}}$ if $r_{1}\le r_{2}$. Note
that $|K_{n,u}f|$ and $|K_{u}f|$ are both bounded, say, by a positive
constant $C_{K}$. For $i\in\BN$, we have
\begin{align*}
 & \left|\mathbf{E}_{i}^{s,x}\Big[Y\int_{t_{1}}^{t_{2}}K_{n,u}f(X_{u})du\Big]-\mathbf{E}_{i}^{s,x}\Big[Y\int_{t_{1}}^{t_{2}}K_{u}f(X_{u})du\Big]\right|\\
 & \ \le C_{Y}\mathbf{E}_{i}^{s,x}\Big[\int_{t_{1}}^{t_{2}}|K_{n,u}f-K_{u}f|(X_{u})du\Big]\\
 & \ \le C_{Y}\mathbf{E}_{i}^{s,x}\Big[\int_{t_{1}}^{t_{2}}(|K_{n,u}f-K_{u}f\ast\phi_{n}|+|K_{u}f\ast\phi_{n}-K_{u}f|)(X_{u})du\Big]\\
 & \ \le C_{Y}\mathbf{E}_{i}^{s,x}\Big[\int_{t_{1}}^{t_{2}}|K_{n,u}f-K_{u}f\ast\phi_{n}|(X_{u})du\Big]+2C_{Y}C_{K}\mathbf{E}_{i}^{s,x}\Big[\int_{t_{1}}^{t_{2}}(1-\chi_{r})(X_{u})du\Big]\\
 & \quad+C_{Y}\mathbf{E}_{i}^{s,x}\Big[\int_{t_{1}}^{t_{2}}\chi_{r}(X_{u})|K_{u}f\ast\phi_{n}-K_{u}f|(X_{u})du\Big]\\
 & \ {\displaystyle \overset{(\ref{eq: kry p_n})}{\le}}t_{2}C_{Y}\|K_{n,u}f-K_{u}f\ast\phi_{n}\|+2C_{Y}C_{K}\sup_{i\in\BN}\mathbf{E}_{i}^{s,x}\Big[\int_{t_{1}}^{t_{2}}(1-\chi_{r})(X_{u})du\Big]\\
 & \quad+C_{Y}C_{t_{2}}\left\Vert \chi_{r}(K_{u}f\ast\phi_{n}-K_{u}f)\right\Vert _{L^{p}([0,t_{2}]\times\Rd)}\\
 & \ =:J_{1}+J_{2}+J_{3}.
\end{align*}
For any given $\epsilon_{1}>0$, by dominated convergence theorem,
we can find sufficiently large $r_{0}>0$ such that
\begin{equation}
\mathbf{E}^{s,x}\Big[\int_{t_{1}}^{t_{2}}(1-\chi_{r_{0}})(X_{u})du\Big]<\epsilon_{1}.\label{esti 1, x_R}
\end{equation}
By the weak convergence of $\mathbf{{P}}_{i}^{s,x}$ to $\mathbf{{P}}^{s,x}$,
we have
\[
\lim_{i\to\infty}\mathbf{E}_{i}^{s,x}\Big[\int_{t_{1}}^{t_{2}}(1-\chi_{r_{0}})(X_{u})du\Big]=\mathbf{E}^{s,x}\Big[\int_{t_{1}}^{t_{2}}(1-\chi_{r_{0}})(X_{u})du\Big].
\]
So there exists $i_{0}\in\BN$ such that
\begin{equation}
\sup_{i>i_{0}}\mathbf{E}_{i}^{s,x}\Big[\int_{t_{1}}^{t_{2}}(1-\chi_{r_{0}})(X_{u})du\Big]\le2\epsilon_{1}.\label{esti 2, x_R}
\end{equation}
Similarly to (\ref{esti 1, x_R}), for $i=1,2,\cdots,i_{0}$, we can
find $r_{1}>r_{0}$ such that
\begin{equation}
\sup_{1\le i\le i_{0}}\mathbf{E}_{i}^{s,x}\Big[\int_{t_{1}}^{t_{2}}(1-\chi_{r_{1}})(X_{u})du\Big]<\epsilon_{1}.\label{esti 3, x_R}
\end{equation}
Combining (\ref{esti 2, x_R}) and (\ref{esti 3, x_R}) and noting
that $\chi_{r}$ is non-decreasing in $r$, we get
\[
\sup_{i\in\BN}\mathbf{E}_{i}^{s,x}\Big[\int_{t_{1}}^{t_{2}}(1-\chi_{r})(X_{u})du\Big]<3\epsilon_{1},\quad r\ge r_{1}.
\]
Hence we have shown that $\lim_{r\to\infty}J_{2}=0.$ By Lemma \ref{lem: delta },
we have $J_{1}\to0$ as $n\to\infty$. It is also easy to see that
$J_{3}\to0$ as $n\to\infty$. With a simple ``$\epsilon-\delta$''-argument,
we obtain
\begin{equation}
\lim_{n\to\infty}\left|\mathbf{E}_{i}^{s,x}\Big[Y\int_{t_{1}}^{t_{2}}K_{n,u}f(X_{u})du\Big]-\mathbf{E}_{i}^{s,x}\Big[Y\int_{t_{1}}^{t_{2}}K_{u}f(X_{u})du\Big]\right|=0,\label{neweqek-1}
\end{equation}
and the convergence is uniform with respect to $i\in\BN.$

Similarly to (\ref{neweqek-1}), we have
\begin{equation}
\lim_{n\to\infty}\left|\mathbf{E}^{s,x}\Big[Y\int_{t_{1}}^{t_{2}}K_{n,u}f(X_{u})du\Big]-\mathbf{E}^{s,x}\Big[Y\int_{t_{1}}^{t_{2}}K_{u}f(X_{u})du\Big]\right|=0.\label{neweqesx}
\end{equation}
By (\ref{neweqek-1}) and (\ref{neweqesx}), for any given $\epsilon>0$,
we can find $n_{1}\in\BN$, which is independent of $i$, such that
for all $n,m\ge n_{1}$ and $i\in\BN$,
\[
\left|\mathbf{E}^{s,x}\Big[Y\int_{t_{1}}^{t_{2}}K_{n,u}f(X_{u})du\Big]-\mathbf{E}^{s,x}\Big[Y\int_{t_{1}}^{t_{2}}K_{u}f(X_{u})du\Big]\right|<\epsilon
\]
and
\[
\left|\mathbf{E}_{i}^{s,x}\Big[Y\int_{t_{1}}^{t_{2}}K_{n,u}f(X_{u})du\Big]-\mathbf{E}_{i}^{s,x}\Big[Y\int_{t_{1}}^{t_{2}}K_{m,u}f(X_{u})du\Big]\right|<\epsilon.
\]
Similarly to (\ref{eq 3, thm 5.3: to show}), there exists $n_{2}\in\BN$
such that for $n\ge n_{2}$,
\[
\left|\mathbf{E}_{n}^{s,x}\Big[Y\int_{t_{1}}^{t_{2}}K_{n_{1},u}f(X_{u})du\Big]-\mathbf{E}^{s,x}\Big[Y\int_{t_{1}}^{t_{2}}K_{n_{1},u}f(X_{u})du\Big]\right|<\epsilon.
\]
If $n\ge\sup\{n_{1},n_{2}\}$, then

\begin{align*}
 & \left|\mathbf{E}_{n}^{s,x}\Big[Y\int_{t_{1}}^{t_{2}}K_{n,u}f(X_{u})du\Big]-\mathbf{E}^{s,x}\Big[Y\int_{t_{1}}^{t_{2}}K_{u}(X_{u})du\Big]\right|\\
\le & \left|\mathbf{E}_{n}^{s,x}\Big[Y\int_{t_{1}}^{t_{2}}K_{n,u}f(X_{u})du\Big]-\mathbf{E}_{n}^{s,x}\Big[Y\int_{t_{1}}^{t_{2}}K_{n_{1},u}f(X_{u})du\Big]\right|\\
 & \ +\left|\mathbf{E}_{n}^{s,x}\Big[Y\int_{t_{1}}^{t_{2}}K_{n_{1},u}f(X_{u})du\Big]-\mathbf{E}^{s,x}\Big[Y\int_{t_{1}}^{t_{2}}K_{n_{1},u}f(X_{u})du\Big]\right|\\
 & \quad\ +\left|\mathbf{E}^{s,x}\Big[Y\int_{t_{1}}^{t_{2}}K_{n_{1},u}f(X_{u})du\Big]-\mathbf{E}^{s,x}\Big[Y\int_{t_{1}}^{t_{2}}K_{u}f(X_{u})du\Big]\right|\le3\epsilon.
\end{align*}
So (\ref{eq 2, thm 5.3: to show}) is true.

``\textit{Step 4}'': We finally prove (\ref{eq, thm 5.3: to show})
under the condition (\ref{extra ass, I}). Since $\mathbf{P}_{n}^{s,x}$
solves the martingale problem for $\mathcal{L}_{n,t}$, it follows
that
\begin{equation}
\mathbf{E}_{n}^{s,x}\Big[Y\big(f(X_{t_{2}})-f(X_{t_{1}})-\int_{t_{1}}^{t_{2}}\mathcal{L}_{u}f(X_{u})du\big)\Big]=0.\label{eqmartingale}
\end{equation}
So (\ref{eq, thm 5.3: to show}) follows from (\ref{eqmartingale}),
(\ref{eq 3, thm 5.3: to show}) and (\ref{eq 2, thm 5.3: to show}).

This completes the proof of existence.

``\emph{Uniqueness}'': Let $(s,x)\in\hs$ be arbitrary and $\mathbf{{\tilde{P}}}^{s,x}$
be a solution to the martingale problem for $\mathcal{L}_{t}$ starting
from $(s,x)$. For each $f\in C_{b}^{1,2}(\mathbb{R}_{+}\times\mathbb{R}^{d})$,
\[
f(t,X_{t})-f(s,X_{s})-{\displaystyle \int_{s}^{t}\left(\frac{\partial f}{\partial u}+\mathcal{L}_{u}f\right)(u,X_{u})du}
\]
is an $\mathcal{F}_{t}$-martingale after $s$ with respect to the
measure $\mathbf{{\tilde{P}}}^{s,x}$. For any $s\leq t_{1}<t,\ C\in\mathcal{F}_{t_{1},}$
we thus have
\begin{equation}
\mathbf{{\tilde{E}}}^{s,x}[\mathbf{{1}}_{C}f(t,X_{t})]=\mathbf{{\tilde{E}}}^{s,x}[\mathbf{{1}}_{C}f(t_{1},X_{t_{1}})]+\mathbf{{\tilde{E}}}^{s,x}\left[\mathbf{{1}}_{C}\int_{t_{1}}^{t}\left(\frac{\partial f}{\partial u}+\mathcal{L}_{u}f\right)(u,X_{u})du]\right].\label{eq5.1:MPAFL}
\end{equation}
Similarly to (\ref{eq: Fubini}), by multiplying both sides of (\ref{eq5.1:MPAFL})
by $\exp(-\lambda(t-t_{1}))$ and then integrating with respect to
$t$ from $t_{1}$ to $\infty$, we get
\begin{align*}
 & \mathbf{{\tilde{E}}}^{s,x}\left[\mathbf{{1}}_{C}{\displaystyle \int_{t_{1}}^{\infty}e^{-\lambda(t-t_{1})}f(t,X_{t})dt}\right]\\
 & \quad=\lambda^{-1}\mathbf{{\tilde{E}}}^{s,x}\left[\mathbf{{1}}_{C}f(t_{1},X_{t_{1}})\right]+\lambda^{-1}\mathbf{{\tilde{E}}}^{s,x}\left[\mathbf{{1}}_{C}{\displaystyle \int_{t_{1}}^{\infty}e^{-\lambda(u-t_{1})}\left(\frac{\partial f}{\partial u}+\mathcal{L}_{u}f\right)(u,X_{u})du}\right].
\end{align*}
 Therefore,
\begin{align}
 & \mathbf{{\tilde{E}}}^{s,x}\left[\int_{t_{1}}^{\infty}e^{-\lambda(t-t_{1})}f(t,X_{t})dt|\mathcal{F}_{t_{1}}\right]\nonumber \\
\text{} & {\displaystyle \quad=\lambda^{-1}f(t_{1},X_{t_{1}})+\lambda^{-1}\mathbf{{\tilde{E}}}^{s,x}\left[\int_{t_{1}}^{\infty}e^{-\lambda(t-t_{1})}\left(\frac{\partial f}{\partial t}+\mathcal{L}_{t}f\right)(t,X_{t})dt|\mathcal{F}_{t_{1}}\right].}\label{eq5.2:MPAFL}
\end{align}

If $g\in C_{b}^{1,2}(\mathbb{R}_{+}\times\mathbb{R}^{d})$, then $f:=R_{\lambda}g\in C_{b}^{1,2}(\mathbb{R}_{+}\times\mathbb{R}^{d})$
and (\ref{eq for para. resol. eq.: MPAFL}) holds. Substituting this
$f$ in (\ref{eq5.2:MPAFL}), we obtain
\begin{align}
 & \mathbf{{\tilde{E}}}^{s,x}\left[\int_{t_{1}}^{\infty}e^{-\lambda(t-t_{1})}g(t,X_{t})dt|\mathcal{F}_{t_{1}}\right]\nonumber \\
 & \quad=R_{\lambda}g(t_{1},X_{t_{1}})+\mathbf{{\tilde{E}}}^{s,x}\left[{\displaystyle \int_{t_{1}}^{\infty}e^{-\lambda(t-t_{1})}KR_{\lambda}g(t,X_{t})dt|\mathcal{F}_{t_{1}}}\right].\label{eq5.3:MPAFL}
\end{align}
With a similar argument as in the proof of (\ref{WUSeqlemma43vnl}),
we see that (\ref{eq5.3:MPAFL}) is true for all $g\in\mathcal{B}_{b}(\mathbb{R}_{+}\times\mathbb{R}^{d})$.

If $g\in\mathcal{B}_{b}(\mathbb{R}_{+}\times\mathbb{R}^{d})$, then
$KR_{\lambda}g\in\mathcal{B}_{b}(\mathbb{R}_{+}\times\mathbb{R}^{d})$.
By (\ref{eq5.3:MPAFL}) and a simple iteration, we obtain for each
$k\in\mathbb{{N}}$,
\begin{align*}
 & \mathbf{{\tilde{E}}}^{s,x}\left[\int_{t_{1}}^{\infty}e^{-\lambda(t-t_{1})}g(t,X_{t})dt|\mathcal{F}_{t_{1}}\right]\\
 & \quad=\sum_{i=0}^{k}R_{\lambda}(KR_{\lambda})^{i}g(t_{1},X_{t_{1}})+\mathbf{{\tilde{E}}}^{s,x}\left[{\displaystyle \int_{t_{1}}^{\infty}e^{-\lambda(t-t_{1})}(KR_{\lambda})^{k+1}g(t,X_{t})dt|\mathcal{F}_{t_{1}}}\right].
\end{align*}
By Proposition \ref{prop MPAFL:For-any-} and Corollary \ref{cor MAPAFL:There-exists-},
we see that
\begin{equation}
\mathbf{{\tilde{E}}}^{s,x}\left[\int_{t_{1}}^{\infty}e^{-\lambda(t-t_{1})}g(t,X_{t})dt|\mathcal{F}_{t_{1}}\right]=\sum_{i=0}^{\infty}R_{\lambda}(KR_{\lambda})^{i}g(t_{1},X_{t_{1}})=G_{\lambda}g(t_{1},X_{t_{1}})\label{eq5.5:MPAFL}
\end{equation}
for all $\lambda\ge\lambda_{0}$ and $g\in\mathcal{B}_{b}(\mathbb{R}_{+}\times\mathbb{R}^{d})$.

Note that the choice of $t_{1}\in[s,\infty)$ in (\ref{eq5.5:MPAFL})
is arbitrary. It follows from (\ref{eq5.5:MPAFL}), Remark \ref{rem: contin. of G_lambda}
and \cite[Lemma~3.1]{MR736974} that there exists at most one solution
to the martingale problem for $\mathcal{L}_{t}$ starting from $(s,x)$.
\qed

\bibliographystyle{amsplain}

\begin{thebibliography}{10}
\bibitem{chen2013uniqueness}
Zhen-Qing Chen and Longmin Wang, \emph{Uniqueness of stable processes with
  drift}, Proc. Amer. Math. Soc. \textbf{144} (2016), no.~6, 2661--2675.
  \MR{3477084}

\bibitem{chen2016uniqueness}
Zhen-Qing Chen and Xicheng Zhang, \emph{Uniqueness of stable-like processes},
  arXiv preprint arXiv:1604.02681 (2016).

\bibitem{MR838085}
Stewart~N. Ethier and Thomas~G. Kurtz, \emph{Markov processes: Characterization
  and convergence}, Wiley Series in Probability and Mathematical Statistics:
  Probability and Mathematical Statistics, John Wiley \& Sons, Inc., New York,
  1986. \MR{838085 (88a:60130)}

\bibitem{jin2015weak}
Peng Jin, \emph{On weak solutions of {SDE}s with singular time-dependent drift
  and driven by stable processes}, Stoch. Dyn., to appear. Available at
  http://arxiv.org/abs/1512.02689.

\bibitem{MR3192504}
Panki Kim and Renming Song, \emph{Stable process with singular drift},
  Stochastic Process. Appl. \textbf{124} (2014), no.~7, 2479--2516.
  \MR{3192504}

\bibitem{MR736974}
Takashi Komatsu, \emph{On the martingale problem for generators of stable
  processes with perturbations}, Osaka J. Math. \textbf{21} (1984), no.~1,
  113--132. \MR{736974 (86e:60060)}

\bibitem{MR3201992}
R.~Mikulevicius and H.~Pragarauskas, \emph{On the {C}auchy problem for
  integro-differential operators in {H}\"older classes and the uniqueness of
  the martingale problem}, Potential Anal. \textbf{40} (2014), no.~4, 539--563.
  \MR{3201992}

\bibitem{MR3145767}
R.~Mikulevi{\v c}ius and H.~Pragarauskas, \emph{On the {C}auchy problem for
  integro-differential operators in {S}obolev classes and the martingale
  problem}, J. Differential Equations \textbf{256} (2014), no.~4, 1581--1626.
  \MR{3145767}

\bibitem{MR1248747}
R.~Mikulyavichyus and G.~Pragarauskas, \emph{The martingale problem related to
  nondegenerate {L}\'evy operators}, Liet. Mat. Rink. \textbf{32} (1992),
  no.~3, 377--396. \MR{1248747}

\bibitem{MR2945756}
Enrico Priola, \emph{Pathwise uniqueness for singular {SDE}s driven by stable
  processes}, Osaka J. Math. \textbf{49} (2012), no.~2, 421--447. \MR{2945756}

\bibitem{MR1739520}
Ken-iti Sato, \emph{L\'evy processes and infinitely divisible distributions},
  Cambridge Studies in Advanced Mathematics, vol.~68, Cambridge University
  Press, Cambridge, 1999, Translated from the 1990 Japanese original, Revised
  by the author. \MR{1739520 (2003b:60064)}

\bibitem{MR0433614}
Daniel~W. Stroock, \emph{Diffusion processes associated with {L}\'evy
  generators}, Z. Wahrscheinlichkeitstheorie und Verw. Gebiete \textbf{32}
  (1975), no.~3, 209--244. \MR{0433614 (55 \#6587)}

\bibitem{MR2190038}
Daniel~W. Stroock and S.~R.~Srinivasa Varadhan, \emph{Multidimensional
  diffusion processes}, Classics in Mathematics, Springer-Verlag, Berlin, 2006,
  Reprint of the 1997 edition. \MR{MR2190038 (2006f:60005)}

\bibitem{MR2286060}
Toshiro Watanabe, \emph{Asymptotic estimates of multi-dimensional stable
  densities and their applications}, Trans. Amer. Math. Soc. \textbf{359}
  (2007), no.~6, 2851--2879 (electronic). \MR{2286060 (2008e:60143)}

\bibitem{MR3127913}
Xicheng Zhang, \emph{Stochastic differential equations with {S}obolev drifts
  and driven by {$\alpha$}-stable processes}, Ann. Inst. Henri Poincar\'e
  Probab. Stat. \textbf{49} (2013), no.~4, 1057--1079. \MR{3127913}

\end{thebibliography}

\end{document}